\documentclass[12pt]{amsart}
\newtheorem{theorem}{Theorem}[section]
\newtheorem{lemma}[theorem]{Lemma}
\newtheorem{corollary}[theorem]{Corollary}
\newtheorem{condition}[theorem]{Condition}
\newtheorem{proposition}[theorem]{Proposition}
\newtheorem{remark}[theorem]{Remark}

\newtheorem{Question}[theorem]{Question}
\newcommand{\ncom}{\newcommand}
\ncom{\lrar}{\longrightarrow}
\ncom{\rar}{\rightarrow}
\ncom{\ov}{\overline}
\ncom{\m}{\mbox}
\ncom{\sta}{\stackrel}
\ncom{\comx}{{\mathbb C}}
\ncom{\A}{{\mathbb A}}
\ncom{\Z}{{\mathbb Z}}
\ncom{\Q}{{\mathbb Q}}
\ncom{\R}{{\mathbb R}}
\ncom{\G}{{\mathbb G}}
\ncom{\al}{\alpha}
\ncom{\p}{{\mathbb P}}
\ncom{\E}{{\mathbb E}}
\ncom{\N}{{\mathbb N}}
\ncom{\K}{{\mathbb K}}
\ncom{\X}{{\mathbb X}}
\ncom{\f}{\frac}
\ncom{\cA}{{\mathcal A}}
\ncom{\cB}{{\mathcal B}}
\ncom{\cD}{{\mathcal D}}
\ncom{\cX}{{\mathcal X}}
\ncom{\cO}{{\mathcal O}}
\ncom{\cW}{{\mathcal W}}
\ncom{\cL}{{\mathcal L}}
\ncom{\cP}{{\mathcal P}}
\ncom{\cH}{{\mathcal H}}
\ncom{\cS}{{\mathcal S}}
\ncom{\cM}{{\mathcal M}}
\ncom{\cC}{{\mathcal C}}
\ncom{\cT}{{\mathcal T}}
\ncom{\cF}{{\mathcal F}}
\ncom{\cN}{{\mathcal N}}
\ncom{\cJ}{{\mathcal J}}
\ncom{\cK}{{\mathcal K}}
\ncom{\cV}{{\mathcal V}}
\ncom{\cZ}{{\mathcal Z}}
\ncom{\cU}{{\mathcal U}}
\ncom{\cSU}{{\mathcal S \mathcal U}}
\ncom{\cG}{{\mathcal G}}
\ncom{\cQ}{{\mathcal Q}}
\ncom{\cR}{{\mathcal R}}
\ncom{\cY}{{\mathcal Y}}
\ncom{\cE}{{\mathcal E}}
\ncom{\what}{\widehat}
\ncom{\delbar}{\overline{\partial}}

\ncom{\eop}{{\hfill $\Box$}}

\begin{document}
\baselineskip=16pt

\title{Regulators of canonical extensions are torsion: the smooth divisor case}


\author[J. N. Iyer]{Jaya NN  Iyer}

\address{School of Mathematics, Institute for Advanced Study, 1 Einstein Drive, Princeton NJ 08540 USA.}
\email{jniyer@ias.edu}

\address{The Institute of Mathematical Sciences, CIT
Campus, Taramani, Chennai 600113, India}
\email{jniyer@imsc.res.in}

\author[C. T.  Simpson]{Carlos T  Simpson}
\address{CNRS, Laboratoire J.-A.Dieudonn\'e, Universit\'e de Nice--Sophia Antipolis,
Parc Valrose, 06108 Nice Cedex 02, France}
\email{carlos@math.unice.fr}

\footnotetext{Mathematics Classification Number: 14C25, 14D05, 14D20, 14D21 }

\begin{abstract}
In this paper, we prove a generalization of Reznikov's theorem which says that 
the Chern-Simons classes and in particular the Deligne Chern classes (in degrees $>\,1$) are torsion, of a 
flat vector bundle on a smooth complex projective variety. We consider the case of a smooth quasi--projective variety 
with an irreducible smooth divisor at infinity. We define
the Chern-Simons classes of the Deligne's \textit{canonical extension} of a flat vector bundle with unipotent monodromy 
at infinity, which lift the Deligne Chern classes and prove that these classes are torsion. 
\end{abstract}
\maketitle

\setcounter{tocdepth}{1}
\tableofcontents


\section{Introduction}

Chern and Simons \cite{Chn-Si} and Cheeger \cite{Ch-Si} introduced a theory of differential cohomology on smooth manifolds. 
For vector bundles with connection, they defined classes or the secondary invariants in the ring of differential characters. 
These classes lift the closed form defined by the curvature form of the given connection.
In particular when the connection is flat, the secondary invariants yield classes in the cohomology with $\R/\Z$-coefficients. 
These are the Chern-Simons classes of flat connections. 

The following question was raised in \cite[p.70-71]{Ch-Si} (see also \cite[p.104]{Bl}) by Cheeger and Simons:

\begin{Question}\label{Cheeger-Simons}
Suppose $X$ is a smooth manifold and $(E,\nabla)$ is a flat connection on $X$.
Are the Chern-Simons classes $\what{c_i}(E,\nabla)$ of $(E,\nabla)$ torsion in $H^{2i-1}(X,\R/\Z)$, for $i\geq 2 $  ?
\end{Question} 

Suppose $X$ is a smooth projective variety defined over the complex numbers. Let $(E,\nabla)$ be a vector bundle with a flat 
connection $\nabla$. 
S. Bloch \cite{Bl} showed that for a unitary connection the Chern-Simons classes are mapped to the
Chern classes of $E$ in the Deligne cohomology. The above Question \ref{Cheeger-Simons} together with his observation led him to 
conjecture that the Chern classes of flat bundles are torsion in the Deligne cohomology of $X$, in degrees at least two.

A. Beilinson defined universal secondary classes and H. Esnault \cite{Es} constructed secondary classes using a modified splitting 
principle in the $\comx/\Z$-cohomology. These classes are shown to be liftings of the Chern classes in the Deligne cohomology. 
These classes also have an interpretation in terms of differential characters, and the original $\R / \Z$ classes of
Chern-Simons are obtained by the projection $\comx / \Z \rightarrow \R / \Z$. The imaginary parts of the 
$\comx / \Z$ classes are Borel's volume regulators $Vol_{2p-1}(E,\nabla )\in H^{2p-1}(X,\R )$. All the constructions give the same class in odd degrees, called as the secondary classes on X (see \cite{DHZ}, \cite{Es2} for a discussion on this).

Reznikov \cite{Re}, \cite{Re2} showed 
 that the secondary classes of $(E,\nabla)$ are torsion in the cohomology $H^{2i-1}(X,\comx/\Z )$ 
of $X$, when $i\geq 2$. In particular, he proved the above mentioned conjecture of Bloch.

Our aim here is to extend this result when $X$ is smooth and quasi--projective with an irreducible smooth divisor $D$ at infinity. We consider a flat bundle on $X$ which has unipotent monodromy at infinity. We define secondary classes on $X$ (extending the classes on $X-D$ of the flat connection) and which lift the Deligne Chern classes, and show that these classes are torsion.

Our main theorem is
\begin{theorem}
Suppose $X$ is a smooth quasi--projective variety defined over $\comx$. Let $(E,\nabla)$ be a flat connection on $U:=X-D$ associated to a representation $\rho:\pi_1(U)\rar GL_r(\comx)$. Assume that $D$ is a smooth and irreducible divisor and $(\ov{E},\ov{\nabla})$ be the Deligne canonical extension on $X$ with unipotent monodromy around $D$. Then 
the secondary classes 
$$
\what{c_p}(\rho / X )\,\in\,H^{2p-1}(\ov X , \comx / \Z )
$$
of $(\ov{E},\ov{\nabla})$ are torsion, for $p>1$. If, furthermore, $X$ is projective then 
the Chern classes of $\ov E$ are torsion in
the Deligne cohomology of $X$, in degrees $>1$.
\end{theorem}

What we do here can easily be generalized to the case when $D$ is smooth and has several
disjoint irreducible components. On the other hand, the generalization to a normal crossings divisor presents
significant new difficulties which we don't yet know how to handle, so this will be left for the future. 

The main constructions in this paper are as follows.
We will consider the following situation. Suppose $X$ is a smooth manifold, and 
$D\subset X$ is a connected smooth closed subset of real codimension $2$. Let $U:= X-D$ and suppose we can choose a reasonable
tubular neighborhood $B$ of $D$. 
Let $B^{\ast} := B\cap U = B-D$.
It follows that $\pi _1(B^{\ast})\rightarrow \pi _1(B)$ is surjective.
 The diagram
\begin{equation}
\label{pushout}
\begin{array}{ccc}
B^{\ast} & \rightarrow & B \\
\downarrow & & \downarrow \\
U & \rightarrow & X
\end{array}
\end{equation}
is a homotopy pushout diagram. Note also that $B$ retracts to $D$, and $B^{\ast}$ has a tubular structure: 
$$
B^{\ast} \cong S \times (0,1)
$$
where $S\cong \partial B$ is a circle bundle over $D$.

We say that $(X,D)$ is {\em complex algebraic} if $X$ is a smooth complex quasiprojective variety and $D$ an irreducible smooth divisor. 

Suppose we are given a representation $\rho : \pi _1(U)\rightarrow GL_r(\comx )$, corresponding to a local system $L$ over $U$,
or equivalently to a vector bundle with flat connection $(E,\nabla )$. Let $\gamma$ be a loop going out from the basepoint to a point near $D$,
once around, and back. Then $\pi _1(B)$ is obtained from $\pi _1(B^{\ast})$ by adding the relation $\gamma \sim 1$. 
We assume that {\em the monodromy of $\rho$ at infinity is unipotent}, by which we mean that $\rho (\gamma )$ should be unipotent.
The logarithm is a nilpotent transformation
$$
N:= \log \rho (\gamma ) := 
(\rho (\gamma )-I) - \f{1}{2}(\rho (\gamma ) -I)^2 + \f{1}{3}(\rho (\gamma )- I)^3-... ,
$$
where the series stops after a finite number of terms. 

In this situation, there is a canonical and natural way to extend the bundle $E$ to a bundle $\ov E$ over $X$, known as the 
{\em Deligne canonical extension} \cite{De}. The connection $\nabla$ extends
to a connection $\ov{\nabla}$ whose singular terms involved look locally like $Nd\theta$ where $\theta$ is the angular coordinate around $D$. 
In an appropriate frame the singularities of $\ov \nabla $ are only in the strict upper triangular region of the
connection matrix.  In the complex algebraic case, $(E,\nabla )$ are holomorphic, and indeed algebraic with algebraic structure
uniquely determined by the requirement that $\nabla$ have regular singularities. The extended bundle $\ov E$ is
algebraic on $X$ and
$\ov\nabla$ becomes a logarithmic connection \cite{De}. 

We will define {\em extended regulator classes} 
$$
\what{c}_p(\rho / X) \in H^{2p-1}(X, \comx / \Z )
$$
which restrict to the usual regulator classes on $U$. Their imaginary parts define {\em extended volume regulators}
which we write as $Vol_{2p-1}(\rho / X)\in H^{2p-1}(X, \R )$. 

The technique for defining the extended regulator classes is to construct a {\em patched connection} $\nabla ^{\#}$ over $X$.
This will be a smooth connection, however it is not flat. Still, the curvature comes from the singularities of 
$\ov \nabla$ which have been smoothed out, so the curvature is upper-triangular. In particular, the Chern forms for 
$\nabla ^{\#}$ are still identically zero. The Cheeger-Simons theory of differential characters provides a
class of $\nabla ^{\#}$ in the group of differential characters, mapping to the group of closed forms. Since the image,
which is the Chern form, vanishes, the differential character lies in the kernel of this map which is exactly
$H^{2p-1}(X, \comx / \Z )$ \cite[Cor. 2.4]{Ch-Si}. This is the construction of the regulator class.

The proof of Dupont-Hain-Zucker that the regulator class lifts the Deligne Chern class, goes through word for word here
to show that this extended regulator class lifts the Deligne Chern class of the canonical extension $\overline{E}$
in the complex algebraic case. For this part, we need $X$ projective. 

We also give a different construction of the regulator classes, using the deformation theorem in $K$-theory.
The filtration which we will use to define the patched connection, also leads to a 
polynomial deformation on $B^{\ast}$ between
the representation $\rho$ and its associated-graded. Then, using the fact that $BGL(F[t])^+$ is homotopy-equivalent
to $BGL(F)^+$ and the fact that the square \eqref{pushout} is a homotopy pushout, this 
allows us to construct a map from $X$ 
to $BGL(F)^+$ and hence pull back the universal regulator classes. 
Corollary \ref{patchingsame} below says that these are the same as the extended
regulators defined by the patched connection.  
On the other hand, the counterpart of the deformation
construction in hermitian $K$-theory
allows us to conclude that the extended volume regulator is zero whenever $\rho$ 
underlies a complex variation of Hodge structure
in the complex algebraic case. 

A rigidity statement for the patched connections is discussed and proved in more generality in  \S \ref{sec-rigidity}.
All of the ingredients of Reznikov's original proof \cite{Re2} are now present for the extended classes,
including Mochizuki's theorem that any representation can be deformed to a complex variation of Hodge structure
\cite{Mochizuki}. Thus we show the generalization of Reznikov's result.

{\Small
\textbf{Acknowledgements}: We thank P. Deligne for having useful discussions. His suggestion to consider a glueing construction of the secondary classes (see \S \ref{logarithmicconnections}) and his letter \cite{De3}, motivated some of the main constructions, and we are thankful to him.
We also thank H. Esnault for explaining some of her constructions in \cite{Es}.
  The first named author is supported by the National Science Foundation (NSF) under agreement No. DMS-0111298. 
}

\section{Idea for the construction of secondary classes}


We begin by recalling the differential cohomology introduced by Chern, Cheeger and Simons \cite{Ch-Si},\cite{Chn-Si}. 
Since we want to look at logarithmic connections, we consider these cohomologies 
on complex analytic varieties and on their smooth compactifications. Our aim is to define secondary
classes in the $\comx/\Z$-cohomology for logarithmic connections which have unipotent monodromy along a smooth
boundary divisor. A glueing construction was suggested by Deligne, which uses glueing of secondary 
classes on the open variety and on a tubular neighbourhood of the boundary divisor. In 
\S \ref{patchedconnection} this will be made precise using a patched connection. 

Let $X$ be a nonsingular variety defined over the complex numbers. 
In the following discussion we will
interchangeably use the notation $X$ for the algebraic variety or the underlying complex analytic space.


\subsection{Analytic differential characters on $X$ \cite{Ch-Si}}\label{diff.char}

Let $S_k(X)$ denote the group of $k$-dimensional smooth singular chains on $X$, with integer coefficients. 
Let $Z_k(X)$ denote the subgroup of cycles.  Let us denote 
$$
S^\bullet(X,\Z):=\m{Hom}_\Z(S_\bullet(X),\Z)
$$
the complex of $\Z$ -valued smooth singular cochains, whose boundary operator is denoted by $\delta$.
The group of smooth diffferential $k$-forms on $X$ with complex coefficients is denoted by $A^k(X)$
and the subgroup of closed forms by $A^k_{cl}(X)$.
Then $A^\bullet(X)$ is canonically embedded in $S^\bullet(X)$, by integrating forms against the smooth
singular chains. In fact, we have an embedding
$$
i_\Z:A^\bullet(X)\hookrightarrow S^\bullet(X,\comx/\Z).
$$

The group of differential characters of degree $k$ is defined as
$$
\widehat{H^{k}}(X,\comx/\Z):=\{(f,\al)\in \m{Hom}_\Z(Z_{k-1}(X),\comx/\Z) 
\oplus A^k(X): \delta(f)=i_\Z(\al) \m{ and } d\al=0 \}.
$$
There is a canonical and functorial exact sequence:
\begin{equation}\label{diffeqn}
0\lrar H^{k-1}(X,\comx/\Z)\lrar \widehat{H^{k}}(X,\comx/\Z)\lrar A^k_{cl}(X,\Z)\lrar 0.
\end{equation}
Here $A^k_{cl}(X,\Z):=\m{ker}(A^k_{cl}(X)\lrar H^k(X,\comx/\Z ))$.
Similarly, one defines the group of differential characters $\what{H^{k}}(X,\R/\Z)$ with $\R/\Z$-coefficients.

For the study of infinitesimal variations of differential characters, we have the following
remark about the tangent space.

\begin{lemma}
\label{tangentdc}
The group of differential characters has the structure of infinite dimensional abelian Lie group. Its tangent space 
at the origin (or by translation, at any point) is naturally identified as
$$
T_ 0 \left( \widehat{H^{k}}(X,\comx/\Z ) \right) = \frac{A^{k-1}(X, \comx )}{dA^{k-2}(X,\comx )}.
$$
\end{lemma}
\begin{proof}
A tangent vector corresponds to a path $(f_t,\alpha _t)$. An element $\beta \in A^{k-1}(X, \comx )$
maps to the path given by $f_t(z):= t \int _z \beta $ and $\alpha _t := t d(\beta )$. 
Looking at the above exact sequence \eqref{diffeqn}, we see that this map induces an isomorphism from 
$A^{k-1}(X, \comx )/ dA^{k-2}(X,\comx )$ to the tangent space of $\widehat{H^{k}}(X,\comx/\Z) $.
\end{proof}


\subsection{Secondary classes and the Cheeger-Chern-Simons classes}\label{Secondary classes}

Suppose $(E,\nabla)$ is a vector bundle with a connection on $X$.
 Then the Chern forms
$$
c_k(E,\nabla) \in A^{2k}_{cl}(X,\Z)
$$
for $0\leq k\leq \m{rank }(E)$, are defined using the universal Weil homomorphism \cite{Chn-Si}.
There is an invariant and symmetric 
polynomial $\cP$ of degree $k$ in $k$ variables on the Lie algebra ${\bf gl}_r$ such that if
$\Omega$ is the curvature of $\nabla$ then $c_k(E,\nabla ) = (-1)^k\cP (\Omega , \ldots , \Omega )$.
When $X_i=X$ for each $i$, then $\cP(X,...,X)= \m{trace}(\wedge^kX)$ (see \cite[p.403]{GriffithsHarris}),
however the wedge product here is taken in the variable $\comx ^r$, not the wedge of forms on the base. 
If $X$ is a diagonal matrix with eigenvalues $\lambda _1,\ldots , \lambda _r$ then
 $\cP(X,...,X) = \sum _I \lambda _{i_1}\cdots \lambda _{i_k}$. We can also express $\cP$ in terms
of the traces of products of matrices. In this expression,
the highest order term of $\cP$ is the symmetrization of $Tr (X_1\cdots X_k)$ multiplied by a constant, the
lower order terms are symmetrizations of $Tr(X_1\cdots X_{i_1})Tr(\cdots ) \cdots Tr(X_{i_a+1}\cdots X_k)$, 
with suitable constant coefficients.

The characteristic classes 
$$
\widehat{c_k}(E,\nabla) \in \widehat{H^{2k}}(X,\comx/\Z)
$$
are defined in \cite{Ch-Si} using a factorization of the universal Weil homomorphism and looking at the universal connections \cite{Narasimhan}. 
These classes are functorial liftings of $c_k(E,\nabla)$.

One of the key properties of these classes is the variational formula in case of a family
of connections. If $\{ \nabla _t\}$ is a $\cC^{\infty}$ family of connections on $E$,
then---refering to Lemma \ref{tangentdc} for the tangent space of the space of differential 
characters---we have the formula
\begin{equation}
\label{varform}
\frac{d}{dt}\widehat{c_k}(E,\nabla _t)  = k\cP (\frac{d}{dt}\nabla _t , \Omega_t, \ldots ,\Omega_t ),
\end{equation}
see \cite[Proposition 2.9]{Ch-Si}. 

If $E$ is topologically trivial, then any connection is connected by a path to the trivial 
connection for which the characteristic class is defined to be zero. The variational formula
thus serves to characterize $\widehat{c_k}(E,\nabla _t)$ for all $t$.

\begin{remark}
If the form $c_k(E,\nabla)$ is zero, then the class $\widehat{c_k}(E,\nabla)$ lies in $ H^{2k-1}(X,\comx/\Z)$. 
If $(E,\nabla)$ is a 
flat bundle, then $c_k(E,\nabla)=0$ and the classes $\widehat{c_k}(E,\nabla)$ are called the {\em secondary 
classes} or {\em regulators} of $(E,\nabla)$. Notice 
that the class depends on the choice of $\nabla$.
We will also refer to these classes as the Chern-Simons classes in $\comx/\Z$-cohomology.
\end{remark}

In the case of a flat bundle, after going to a finite cover the bundle is topologically trivial
by the result of Deligne-Sullivan which will be discussed in \S \ref{trivcanext} below. Thus, at least the
pullback to the finite cover of $\widehat{c_k}(E,\nabla)$ can be understood using the variational
methods described above. 

Beilinson's theory of universal secondary classes yield classes for a flat connection $(E,\nabla)$,
\begin{equation}\label{universal}
\widehat{c_k}(E,\nabla) \in H^{2k-1}(X,\comx/\Z),\,\,k\geq 1
\end{equation}
which are functorial and additive over exact sequences.
Furthermore, Esnault \cite{Es} using a modified splitting principle, Karoubi \cite{Ka2} using $K$-theory have 
defined secondary classes. These classes are functorial and additive.  These classes then agree 
with the universally defined class in \eqref{universal} (see \cite[p.323]{Es}). 

When $X$ is a smooth projective variety,
Dupont-Hain-Zucker \cite{Zu}, \cite{DHZ} and Brylinski \cite{Br} have shown that the Chern--Simons classes are liftings of the Deligne Chern class $c_k^\cD(E)$ under the 
map obtained by dividing out by the Hodge filtered piece $F^k$,
$$
H^{2k-1}(X,\comx/\Z)\lrar H^{2k}_\cD(X,\Z(k)).
$$
By functoriality and additive properties, the classes in \eqref{universal} lift the Chern-Simons 
classes defined above using differential characters, via the projection
$$
\comx/\Z\rar \R/\Z.
$$ 
In fact, Cheeger-Simons explicitly took the real part in their formula at the start of \cite[\S 4]{Ch-Si}. 
See also \cite{Bl} for unitary connections, \cite{Soule}, \cite{GilletSoule} when $X$ is smooth and 
projective; for a discussion on this see \cite{Es2}.


\subsection{Secondary classes of logarithmic connections}\label{logarithmicconnections}


Suppose $X$ is a nonsingular variety and $D\subset X$ an irreducible smooth divisor. Let $U:= X-D$. 
Choose a tubular neighborhood $B$ of $D$ and let $B^{\ast}:= B\cap U = B-D$. 

 Let $(E,\nabla)$ be a complex analytic vector bundle on
$U$ with a connection $\nabla$. 
Consider a logarithmic extension $(\ov{E},\ov{\nabla})$ (see \cite{De}) on $X$ of the connection $(E,\nabla)$.
Assuming that the residues are nilpotent, we want to show that the classes $\widehat{c_k}({E},\nabla)\in H^{2i-1}(U,\comx/\Z)$ 
extend on $X$ to give classes in the cohomology with $\comx/\Z$-coefficients which map to the Deligne Chern class of $\ov{E}$.

We want to use the Mayer-Vietoris sequence (a suggestion from Deligne) to motivate a construction of 
secondary classes in this situation. The precise construction will be carried out in \S \ref{patchedconnection}.

Consider the residue transformation
$$
\eta:\ov{E}\lrar \ov{E}\otimes \Omega_{X}(\m{log} D) \sta{res}{\lrar} \ov{E}\otimes\cO_{D}.
$$  
By assumption $\eta$ is nilpotent and let $r$ be the order of $\eta$.

Consider the {\em Kernel filtration} of $\ov{E}_{D}$ induced by the kernels of the operator 
$\eta$:
$$
0=W_{0,D}\subset W_{1,D} \subset W_{2,D} \subset ...\subset W_{r,D}=\ov{E}_{D}.
$$
Here 
$$
W_{j,D}:= \m{kernel}(\eta^{\circ\,\, j}: \ov{E}_{D}\lrar \ov{E}_{D}).
$$
Denote the graded pieces
$$
\textbf{Gr}_j(\ov{E}_D):= W_{j,D}/W_{j+1,D}
$$
and the associated graded
$$
\textbf{Gr}(\ov{E}_D):= \oplus_{j=0}^{r-1}\textbf{ Gr}_j(\ov{E}_D).
$$

\begin{lemma}\label{tube}
Each graded piece $\textbf{Gr}^j(\ov{E}_D)$ (for $0\leq j< r$) is endowed with a flat connection along $D$.
Furthermore, the filtration of $\ov{E}_{D}$ by $W_{j,D}$ extends to a filtration of 
$\ov{E}$ by holomorphic subbundles $W_r$ defined in a tubular neighborhood $B$ of the divisor $D$. 
On $B^{\ast}$ these subbundles are preserved by the connection $\nabla$, and $\nabla$ induces on each graded piece
$\textbf{Gr}_j(\ov{E}_{B^{\ast}})$ a connection which extends to a flat connection over $B$, and induces
the connection mentioned in the first phrase, on $\textbf{Gr}^j(\ov{E}_D)$.
\end{lemma}
\begin{proof}
Suppose $n$ is the dimension of the variety $X$. 
Consider a product of $n$-open disks $\Delta^n$ with coordinates $(t_1,t_2,...,t_n)$ 
around a point of the divisor $D$ so that $D$ is locally defined by $t_1=0$.
Let $\gamma$ be the generator of the fundamental group of the punctured disk 
$\Delta^n- \{t_1=0\}$. Then $\gamma$ is the monodromy 
operator acting on a fibre $E_t$, for $t\in \Delta^n-\{t_1=0\}$.
The operator
$$
N=\m{log } \gamma = (\gamma-I) - \f{1}{2}(\gamma -I)^2 + \f{1}{3}(\gamma - I)^3-...
$$
is nilpotent since by assumption the local monodromy $\gamma$ is unipotent. Further, the order of unipotency of 
$\gamma$ coincides with the order of nilpotency of $N$.
Consider the filtration on the fibre $E_t$ induced by the operator $N$:
$$
0=W^0(t)\subset W^1(t)\subset ...\subset W^{r}(t)=E_t.
$$
such that
$$
W^j(t):= \m{kernel}(N^{j}:E_t\lrar E_t).
$$
Denote the graded pieces
$$
\textbf{gr}^j_t:=W^j(t)/W^{j+1}(t).
$$
Then we notice that the operator $N$ acts trivially on the graded pieces $\textbf{gr}^j_t$. This means 
that $\gamma$ acts  as identity on $\textbf{gr}^j_t$. In other words, $\textbf{gr}^j_t$ (for 
$t \in \Delta^n$) forms a local system on $\Delta^n$ and extends as a local system $\textbf{gr}^j$ in a 
tubular neighbourhood $B$ of $D$ in ${X}$. 

The operation of $\gamma$ around $D$ can be extended to the boundary (see \cite{De} or 
\cite[c) Proposition]{Es-Vi}). More precisely, the operation $\gamma$ (resp. $N$) extends to the sheaf $\ov{E}$ 
and defines 
an endomorphism $\ov{\gamma}$ (resp.$\ov{N}$) of $\ov{E}_{D}$ such that
$$
\m{exp}(-2\pi i.\eta)=\ov{\gamma}_D.
$$
This implies that the kernels defined by the residue transformation $\eta$ and $\ov{N}$ are the same over $D$. 
The graded piece $\textbf{Gr}^j$ is the bundle associated to the local system $\textbf{gr}^j$ in a 
tubular neighbourhood 
$B$ of $D$ in ${X}$. 
\end{proof}
 
\begin{corollary}\label{Kgroup} If $(E_B,\nabla_B)$ denotes the restriction of $(\ov{E},\ov{\nabla})$ on the tubular neighbourhood $B$, then
in the $K_0$-group $K_{an}(B)$ of analytic vector bundles, we have the equality
$$
{E}_{B}= \textbf{Gr}(E_B)=\oplus_j\textbf{Gr}^j({E}_B).
$$
\end{corollary}
\eop

\begin{corollary}\label{sclasses} 
We can define the secondary classes of the restriction $({E}_{B},\nabla_{B})$ to be
$$
\widehat{c}_i({E}_{B},\nabla_B):= \widehat{c}_i(\textbf{Gr}(E_B))
$$ 
in $H^{2i-1}(B,\comx/\Z)$.
\end{corollary}
\eop

For the above construction, we could have replaced the kernel filtration by Deligne's {\em monodromy weight filtration}
$$
0 = W_{-r-1} \subset \ldots \subset W_r = E
$$
or indeed by any filtration of the flat bundle $(E_{B^{\ast}}, \nabla _{B^{\ast}})$
satisfying the following condition: we say that $W_{\cdot}$ is {\em graded-extendable} if 
it is a filtration by flat subbundles or equivalently by sub-local systems, and if 
each associated-graded piece $Gr^W_j$ corresponds to a local system which extends from $B^{\ast}$ to
$B$.


Consider a tubular neighbourhood $B$ of $D$, as obtained in Lemma \ref{tube}, and 
$B^{\ast}:= B\cap U = B-D$. Associate the Mayer-Vietoris sequence for the pair $(U,B)$:
{\small
\begin{eqnarray*}\label{MayerVietoris}
H^{2i-2}(B^{\ast} ,\comx/\Z)\rar H^{2i-1}(X,\comx/\Z)\rar H^{2i-1}(B,\comx/\Z)\oplus H^{2i-1}(U,\comx/\Z)\\
\rar H^{2i-1}(B^{\ast},\comx/\Z)\rar.
\end{eqnarray*}
}
Consider the restrictions $(E_B,\nabla_B)$ on $B$ and $(E,\nabla)$ on $U$.
Then we have the secondary classes, defined in Corollary \ref{sclasses}, 
\begin{equation}\label{glueU}
\what{c_i}(E_B,\nabla_B)\in H^{2i-1}(B,\comx/\Z)
\end{equation}
and 
\begin{equation}\label{glueX}
\what{c_i}(E,\nabla)\in H^{2i-1}(U,\comx/\Z)
\end{equation}
such that
$$
\what{c_i}(E_B,\nabla_B)|_{B^{\ast}}=\what{c_i}(E,\nabla)|_{B^{\ast}}\in H^{2i-1}(B^{\ast},\comx/\Z).
$$
The above Mayer-Vietoris sequence 
yields a class
\begin{equation}\label{mayervietoris}
\what{c}_i(\ov{E},\ov\nabla)\in H^{2i-1}(X,\comx/\Z)
\end{equation}
which is obtained by glueing the classes in \eqref{glueU} and \eqref{glueX}.

As such, the Mayer-Vietoris sequence doesn't uniquely determine the class: there is a possible
indeterminacy by the image of $H^{2i-2}(B^{\ast}, \comx / \Z )$ under the connecting map. 
Nonetheless, we will show in \S \ref{patchedconnection}, using a patched connection, that there is a 
canonically determined class $\what{c_i}(\ov{E},\ov\nabla)$ as above
which is functorial and additive (\S \ref{sec-rigidity}) and moreover it lifts the Deligne Chern class
(\S \ref{delignecompatibility}).


\section{The $\cC^\infty$-trivialization of canonical extensions}\label{trivcanext}


To further motivate the construction of regulator classes, we digress for a moment to 
give a generalization of the result of 
Deligne and Sullivan on topological triviality of flat bundles, to the case of the canonical extension. 
The topological model of the canonical extension we obtain in this section, on an idea communicated to us 
by Deligne \cite{De3}, motivates the construction 
of a filtration triple in \S \ref{deformation-filtration} which is required to define regulator 
classes using $K$-theory. 

Suppose $X$ is a proper $\cC^\infty$-manifold of dimension $d$. Let $E$ be a complex vector bundle
of rank $n$.
It is well-known that if $N\geq \f{d}{2}$, then the Grassmanian manifold $\m{Grass}(n, \comx^{n+N})$ 
of $n$-dimensional subspaces of $\comx^{n+N}$, 
classifies complex vector bundles of rank $n$ on manifolds of dimension $\leq d$. In other words, 
given a complex vector bundle $E$ on $X$, there exists
a morphism
$$
f:X\lrar \m{Grass}(n,\comx^{n+N})
$$
such that the pullback $f^*\cU$ of the tautological bundle $\cU$ on $\m{Grass}(n,\comx^{n+N})$ is $E$.
If the morphism $f$ is homotopic to a constant map then $E$ is trivial as a $\cC^\infty$-bundle.
This observation is used to obtain an upper bound for the order of torsion of Betti Chern classes of flat bundles.

 
\subsection{$\cC^\infty$-trivialization of flat bundles}

Suppose $E$ is equipped with a flat connection $\nabla$. Then the Chern-Weil theory implies that the Betti Chern classes $c^B_i(E) \in H^{2i}(X,\Z)$ are torsion.
An upper bound for the order of torsion was given by Grothendieck \cite{Grothendieck}.
An explanation of the torsion-property is given by the following theorem due to Deligne and Sullivan:

\begin{theorem}\cite{De-Su}\label{DeligneSullivan}
Let $V$ be a complex local system of dimension $n$ on a compact polyhedron $X$ and $\cV=V\otimes \cO_X$ be the corresponding 
flat vector bundle. There exists a finite surjective covering $\pi:\tilde{X}\lrar X$ of $X$ such that the pullback vector 
bundle $\pi^*\cV$ is trivial as a $\cC^\infty$-bundle.
\end{theorem}

An upper bound for the order of torsion is also prescribed in their proof which depends on the field of definition of the monodromy representation.


\subsection{$\cC^\infty$-trivialization of canonical extensions}\label{can.ext}


Suppose $X$ is a complex analytic variety $D\subset X$ a smooth irreducible divisor, and put $U:= X-D$. 
Consider a flat vector bundle $(E,\nabla)$ on $U$ and its canonical extension $(\ov{E},\ov\nabla)$ on 
$X$. Assume that the residues of $\ov\nabla$ are nilpotent.
Then a computation of the de Rham Chern classes by Esnault \cite[Appendix B]{Es-Vi} shows that
these classes are zero. This implies that the Betti Chern classes of $\ov E$ are torsion.
We want to extend the Deligne-Sullivan theorem in this case, reflecting the torsion property of the Betti Chern classes.

\begin{proposition}\label{trivext}
Let $E$ be a flat vector bundle on $U= X- D$, with unipotent monodromy around $D$. 
There is a finite covering $\tilde{U}\lrar U$ 
such that if $\tilde{ X}$ is the normalization of $X$ in
$\tilde{U}$, then the canonical extension of $\pi^*E$ to $\tilde{X}$ is trivial as a $\cC^\infty$-bundle.
\end{proposition}

Note, in this statement, that the normalization $\tilde{X}$ is smooth, and the ramification of the map
$\tilde{X}\rightarrow X$ is topologically constant along $D$. 

The following proof of this proposition is due to Deligne and we reproduce it from \cite{De3}.

Given a flat connection $(E,\nabla)$ on $U$ with unipotent monodromy along $D$, by Lemma \ref{tube}, 
there is a vector bundle $\cF^r$ with a 
filtration on a tubular neighbourhood $B$ of $D$:
$$
(0)= \cF^0\subset \cF^1\subset...\subset \cF^r = \ov{E}|_B
$$
such that the graded pieces are flat connections associated to local systems $V_i$.  

Suppose the monodromy representation of $(E,\nabla)$ is given by
$$
\rho: \pi_1(X)\lrar GL(A)
$$
where $A\subset \comx$ is of finite type over $\Z$. The filtration of the previous paragraph
is also a filtration of
local systems of $A$-modules over $B^{\ast}$. Then the canonical extension itself should 
be trivial as soon as for two maximal ideals $q_1, q_2$ of $A$ having distinct residue field 
characteristic, $\rho$ is trivial mod $q_1$ and $q_2$.
Consider a finite \'etale cover
\begin{equation}\label{cover}
\pi':{U'}\lrar U
\end{equation}
corresponding to the subgroup of $\pi_1(U,u)$ formed of elements $g$ such that $\rho(g)\equiv 1$, mod $q_1$ and mod $q_2$. The index of 
this subgroup divides the order of $GL_r(A/q_1)\times GL_r(A/q_2)$ (see 
\cite{De-Su}). 
Construct a further cover 
$$
\pi:\tilde{U}\lrar U'\lrar U 
$$
such that the filtration and local systems $V_i$ are constant mod $q_1$ and mod $q_2$.

The proof of Proposition \ref{trivext} now follows from a topological result which we formulate as follows.
Suppose a polytope $X$ is the union of polytopes $U$ and $B$, intersecting along $B^{\ast}$.
Suppose we are given:
\newline
(1) On $U$, there is a flat vector bundle $\cV$ coming from a local system $V_A$ of free $A$-modules of rank $n$.
\newline
(2) a filtration $F$ of $V_A$ on $B^{\ast}$ such that the graded piece $\textbf{gr}^i_F$  is a local system of free $A$-modules of rank $n_i$.
\newline
(3) local systems $V_A^i$ on $B$ extending the $\textbf{gr}^i_F$ on $B^{\ast}$.

Suppose these data are trivial mod $q_1, q_2$, i.e., we have constant $V_A$, constant filtration and 
constant extensions.

 From $(V_A,F,V_A^i)$ we get using the embedding $A\subset \comx$ a flat vector bundle $\cV$, a 
filtration $F$ and extensions $\cV^i$. One can use these to construct a vector bundle on $X$ (no longer flat), 
unique up to non-unique isomorphisms as follows: on $B^{\ast}$ pick a vector bundle splitting of the filtration 
and use it to glue to form a  vector bundle $\ov\cV$ on $X$. This should be the topological translation of
``canonical extension''.

\begin{lemma}\label{model}
In the above situation, the vector bundle $\ov\cV$ is trivial.
\end{lemma}  
\begin{proof}
As in \cite{De-Su}, one constructs algebraic varieties
$$
U_1\cap B_1\hookrightarrow U_1,\, U_1\cap B_1\hookrightarrow B_1
$$ 
over $\m{Spec } (\Z)$, which are unions of affine spaces, with the homotopy of
$$
U\cap B \hookrightarrow U,\, U\cap B\hookrightarrow B.
$$
In $(U_1\cup B_1)\times \A^1$, let us take the closed subscheme
$$
(U_1\times \{1\})\cup ((U_1\cap B_1)\times \A^1)\cup (B_1\times \{0\}).
$$
This is a scheme over $\m{Spec }\Z$.

Over $\m{Spec }A$, our data gives a vector bundle $\widetilde{\cV}$:
on $U_1$, given by $V_A$, on $B_1$ by $\oplus V_A^i$, on $(U_1\cap B_1)\times \A^1$ by an interpolation of them: given by the 
subcoherent sheaves $\sum t^{i}.F^i$ of the pullback of $V_A$ (deformation of a filtration to a grading).
More precisely, on $(U_1\cap B_1)\times \A^1$, we consider the coherent subsheaf
$$
\sum _{i} t^i \cdot F_i \subset A[t]\otimes V_A.
$$
It is locally free over $(U_1\cap B_1)\times \A^1$, so it corresponds to a vector bundle. 
When $t=1$, on $U_1\times \{1\}$, this yields the vector bundle given by $V_A$. 
When $t=0$, on $B_1\times\{0\}$, we get the associated graded vector bundle of the
filtration $F$ on $B_1\times \{0\}$.

If we extend scalars to $\comx$, we obtain yet another model $\widetilde{\cV}_{\comx}$
of the canonical extension.
Now mod $q_1, q_2$, we obtain a trivial bundle and the arguments in \cite{De-Su} apply.
Indeed, consider the classifying map 
$$
f:X\rar \m{Grass}(n,\comx^{n+N})
$$
such that the universal bundle on the Grassmanian pulls back to the vector bundle $\widetilde{\cV}$ on $X$.
Here $\m{dim}X= d$ and $N\geq \f{d}{2}$.
Consider the fibre space $X'\rar X$ whose fibre at $x\in X$ is the space of linear embeddings of the 
vector space $\widetilde{\cV}_x$ in $\comx^{n+N}$. The problem is reduced to 
showing that the classifying map $f':X'\rar \m{Grass}(n,\comx^{n+N})$ composed 
with the projection to the $d$-th coskeleton of the Grassmanian is homotopically trivial. 
Since the Grassmanian is simply connected, by Hasse principle for morphisms \cite{Sullivan}, it follows that the 
above composed map is homotopically trivial if and only if for all $l$ the  
$l$-adic completions 
$$
f'_{\hat l}:X'_{\hat l}\rar cosq_d(\m{Grass}(n,\comx^{n+N}))_{\hat l}=cosq_d((\m{Grass}(n,\comx^{n+N})_{\hat l})
$$
are homotopically trivial. Since, there is a maximal ideal $q$ of $A$ whose residue field 
characteristic 
is different from $l$ and such that $\rho$ and the local systems $V_i$ and filtrations are trivial 
mod $q$, the bundle $\widetilde{\cV}$ is trivial  mod $q$.
The lemma from \cite[Lemme]{De-Su} applies directly to conclude that $f'_{\hat l}$ is homotopically 
trivial.  This concludes the lemma.
\end{proof}


\section{Patched connection on the canonical extension}\label{patchedconnection}


The basic idea for making canonical the lifting in \eqref{mayervietoris} is to patch together 
connections sharing the same block-diagonal part,
then apply the Chern-Simons construction to obtain a class in the group of differential 
characters. The
projection to closed forms is zero because the Chern forms of a connection with strictly 
upper-triangular curvature are zero. 
Then, the resulting secondary class is in the kernel in the exact sequence \eqref{diffeqn}.

In this section we will consider a somewhat general open covering situation. However, much
of this generality is not really used in  our main construction of \S \ref{patchsmoothdiv}
where $X$ will be covered by only two open sets and the filtration is trivial on one of them. We hope
that the more general formalism, or something similar, will be useful for the normal-crossings case in the future.


\subsection{Locally nil-flat connections}\label{nil-flat}

Suppose we have a manifold $X$ and a bundle $E$ over $X$,
provided with the following data of local filtrations and connections: we are given a covering of $X$ by open sets $V_i$, and for each $i$
an increasing filtration $W^i$ of the restricted bundle $E|_{V_i}$ by strict subbundles; and furthermore on the associated-graded 
bundles $Gr^{W^i}(E|_{V_i})$ we are given flat connections $\nabla _{i,{\rm Gr}}$. We don't for the moment assume any compatibility between
these for different neighborhoods. Call $(X,E,\{ (V_i,W^i, \nabla _{i, {\rm Gr}})\} )$ a {\em pre-patching collection}. 

We say that a connection $\nabla$ on $E$ is {\em compatible with the pre-patching collection} if on each $V_i$, $\nabla $ preserves
the filtration $W^i$ and induces the flat connection $\nabla _{i, {\rm Gr}}$ on the associated-graded $Gr^{W^i}(E|_{V_i})$.

\begin{proposition}
\label{nildef}
Suppose $(X,E, \nabla )$ is a connection compatible with a 
pre-patching collection $(X,E,\{ (V_i,W^i, \nabla _{i, {\rm Gr}})\} )$. Then:
\newline
(a)\, The curvature form $\Omega$ of $\nabla$ 
is strictly upper triangular with respect to the filtration $W^i$ over each neighborhood $V_i$; 
\newline
(b)\, In particular
if $\cP$ is any invariant polynomial of degree $k$ then $\cP (\Omega , \ldots , \Omega ) = 0$,
for example $Tr (\Omega \wedge \cdots \wedge \Omega ) = 0$; and
\newline
(b)\, The Chern-Simons class of $\nabla$ defines a class $\widehat{c_p}(E,\nabla ) \in 
H^{2p-1}(X, \comx / \Z )$.
\end{proposition}
\begin{proof}
(a): On $V_i$ the connection preserves $W_i$ and induces a flat connection on the graded pieces. This implies exactly that
$\Omega$ is strictly upper-triangular with respect to $W_i$, that is to say that as an $End(E)$-valued $2$-form
we have $\Omega : W^i_k\rightarrow A^2(X,  W^i_{k-1})$.
\newline
 (b): It follows immediately that $Tr (\Omega \wedge \cdots \wedge \Omega ) = 0$,
and the other invariant polynomials are deduced from these by polynomial operations so they vanish too.
\newline
(c): The Chern-Simons class of $\nabla$ projects to zero in $A^k_{cl}(X,\Z )$ by (b), 
so by the basic exact sequence \eqref{diffeqn} it defines
a class in $H^{2p-1}(X, \comx / \Z )$.
\end{proof}

A fundamental observation about this construction is that the class $\widehat{c_p}(E,\nabla )$ depends only
on the pre-patching collection and not on the choice of $\nabla $.

\begin{lemma}
\label{nilinv}
Suppose $(X,E,\{ (V_i,W^i, \nabla _{i, {\rm Gr}})\} )$ is a given pre-patching collection, and suppose $\nabla _0$ and
$\nabla _1$ are connections compatible with this collection. Then the Chern-Simons classes are equal:
$$
\widehat{c_p}(E,\nabla _0 ) =  \widehat{c_p}(E,\nabla _1 ) \;\; \mbox{in}\;\; 
H^{2p-1}(X, \comx / \Z ).
$$
\end{lemma}
\begin{proof}
Choose any affine path $\nabla_t$ of compatible connections between $\nabla_0$ and $\nabla_1$.
For $t=0,1$ this coincides with the previous ones. Let $\Omega _t$ denote the curvature form of $\nabla _t$ and
let $\nabla ' _t$ denote the derivative with respect to $t$. 

By Lemma \ref{tangentdc}, note that the tangent space to the group of differential characters (at any point) is given by 
$$
T(\what{H^{2p}}(X,\comx/\Z)) = \frac{A^{2p-1}(X)}{dA^{2p-2}}.
$$
With respect to this description of the tangent spaces, the derivative of the Chern-Simons class is
given by 
$$
p\cP (\nabla '_t , \Omega _t, \ldots , \Omega _t)= p Tr (\nabla '_t \wedge \Omega _t ^{p-1}) + \ldots .
$$
See \S \ref{Secondary classes}, also \cite[Proposition 2.9]{Ch-Si}. 

On any local neighborhood $V_i$, note that $\nabla _t$ preserves the filtration $W^i$, and
induces the original flat connection on $Gr^{W_i}$; hence for all $t$, $\Omega _t$ and $\nabla '_t$ 
are strictly upper triangular. 
It follows that $Tr (\nabla '_t \wedge \Omega _t ^{a-1}) = 0$ and $Tr (\Omega _t^b) = 0$ so 
all the terms in $p\cP (\nabla '_t , \Omega _t, \ldots , \Omega _t)$ vanish (see \cite[p.403]{GriffithsHarris} for the explicit formula of $\cP$). 
By the variational formula \eqref{varform},
the class in $\what{H^{2p}}(X,\comx/\Z)$ defined by $\nabla _t$
is independant of $t$. In other words, the $\nabla _t$ all define the same class in $H^{2p-1}(X, \comx / \Z )$. 
\end{proof}

Say that a bundle with connection $(X,E,\nabla )$ is {\em locally nil-flat} if there exists a pre-patching collection for which
$\nabla$ is compatible. On the other hand, say that a pre-patching collection
$(X,E,\{ (V_i,W^i, \nabla _{i, {\rm Gr}})\} )$ is a {\em patching collection} if there exists at least one compatible connection.
Any compatible connection will be called a {\em patched connection}.

The above Proposition \ref{nildef} and Lemma \ref{nilinv} say that if $(X,E,\nabla )$ is a locally nil-flat 
connection, then we get a 
Chern-Simons class $\widehat{c_p}(E,\nabla )$, and similarly given a patching collection we get a 
class defined as the class associated to any compatible connection; and these classes are all the same 
so they only depend
on the patching collection so they could be denoted by
$$
\widehat{c_p}(X,E,\{ (V_i,W^i, \nabla _{i, {\rm Gr}})\} ) \in H^{2p-1}(X, \comx /\Z ).
$$ 


\subsection{Refinements}


If we are given a filtration $W_k$ of a bundle $E$ by strict subbundles, 
a {\em refinement} $W'_m$ is another filtration by strict subbundles such that for any $k$ 
there is an $m(k)$ such that $W_k = W'_{m(k)}$. In this case, $W'$ induces a filtration $Gr^W(W')$ on $Gr^W(E)$. 
It will be useful to have a criterion for when two filtrations admit a common refinement. 

\begin{lemma}
\label{commonrefine}
Suppose $E$ is a $\cC^{\infty}$ vector bundle over a manifold, and $\{ U_k \} _{k\in K}$ is a finite collection of strict subbundles containing $0$ and $E$. 
Then it is the collection of bundles 
in a filtration of $E$, if and only if the following criterion is satisfied: for all $j,k\in K$ either $U_k \subset U_j$ or $U_j\subset U_k$.
Suppose $\{ W_i \} _{i\in I}$ and $\{ U_k \} _{k\in K}$ are two filtrations of $E$. Then they admit a common refinement if and only
if the following criterion is satisfied: for any $i\in I$ and any $k\in K$, either $W_i\subset U_k$ or else $U_k \subset W_i$.
\end{lemma}
\begin{proof}
We prove the first part. If the collection corresponds to a filtration then it obviously satisfies 
the criterion. Suppose given a collection of strict subbundles satisfying the criterion.  
The relation $i\leq j \Leftrightarrow U_i \subset U_j$ induces a total order on $K$, and with respect to this total order the 
collection is a filtration.

Now the second part of the lemma follows immediately from the first: the two filtrations admit a common refinement if and only if
the union of the two collections satisfies the criterion of the first part. Given that $\{ W_i \} _{i\in I}$ and $\{ U_k \} _{k\in K}$
are already supposed to be filtrations, they already satisfy the criterion separately. The only other case is when $i\in I$ and $k\in K$ 
which is precisely the criterion of this part. 
\end{proof}

\begin{corollary}
\label{multiplerefine}
Suppose $E$ is a bundle with $N$ filtrations, every two of which admit a common refinement. Then 
the $N$ filtrations admit a common refinement. Furthermore there exists a common refinement in which each component bundle
comes from at least one of the original filtrations. 
\end{corollary}
\begin{proof}
The union of the three collections satisfies the criterion of the first part of Lemma \ref{commonrefine}, since that criterion 
only makes reference to two indices at at time. This union satisfies the condition in the last sentence. 
\end{proof}

A {\em refinement} of a pre-patching collection is a refinement $\tilde{V}_j$ of the open covering, 
with index set $J$ mapping to the
original index set $I$ by a map denoted $j\mapsto i(j)$, 
and open subsets $\tilde{V}_j \subset V_{i(j)}$ such that the $\tilde{V}_j$ still cover $X$.
Plus, on each $\tilde{V}_j$ a filtration $\tilde{W}^j$ of $E|_{\tilde{V}_j}$ which is 
a refinement of the restriction of $W^{i(j)}$
to $\tilde{V}_j$. Finally we assume that over $\tilde{V}_j$ the connection 
$\nabla _{i(j),{\rm Gr}}$ on ${\rm Gr}^{W^i}(E)$ 
preserves the induced filtration ${\rm Gr}^{W^{i(j)}}(\tilde{W}^j_{\cdot})$ and the refined connection 
$\tilde{\nabla}_{j,{\rm Gr}}$ is the connection which is induced by $\nabla _{i(j),{\rm Gr}}$ 
on the associated-graded  ${\rm Gr}^{\tilde{W}_j}(E)$. 

\begin{lemma}
Suppose $\nabla$ is a patched connection compatible with a pre-patching collection 
 $(X,E,\{ (V_i,W^i, \nabla _{i, {\rm Gr}})\} )$, and suppose
$(X,E,\{ (\tilde{V}_j,\tilde{W}^j, \tilde{\nabla} _{j, {\rm Gr}})\} )$ is a refinement for $j\mapsto i(j)$. 
Then $\nabla$ is also a patched connection compatible with $(X,E,\{ (\tilde{V}_j,\tilde{W}^j, \tilde{\nabla} _{j, {\rm Gr}})\} )$.
\end{lemma}
\begin{proof}
The connection $\nabla$ induces on ${\rm Gr}^{W^i}(E)$ the given connection $\nabla _{i,{\rm Gr}}$.
By the definition of refinement, this connection in turn preserves the 
induced filtration ${\rm Gr}^{W^{i(j)}}(\tilde{W}^j_{\cdot})$. It follows that 
$\nabla$ preserves $\tilde{W}^j$. Furthermore, $\nabla _{i,{\rm Gr}}$ induces
on ${\rm Gr}^{\tilde{W}_j}(E)$ the connection $\tilde{\nabla}_{j,{\rm Gr}}$ in the data of the
refinement, and since $\nabla$ induced $\nabla _{i,{\rm Gr}}$  it follows that $\nabla$
induces $\tilde{\nabla}_{j,{\rm Gr}}$ on ${\rm Gr}^{\tilde{W}_j}(E)$.
\end{proof}

\begin{corollary}
\label{commonrefinement}
If two patching collections 
$$
(X,E,\{ (V_i,W^i, \nabla _{i, {\rm Gr}})\} ) \m{  and } 
(X,E,\{ (\tilde{V}_j,\tilde{W}^j, \tilde{\nabla} _{j, {\rm Gr}})\} )
$$ 
admit a common refinement, then 
$$
\widehat{c_p}(X,E,\{ (V_i,W^i, \nabla _{i, {\rm Gr}})\} ) = 
\widehat{c_p}(X,E,\{ (\tilde{V}_j,\tilde{W}^j, \tilde{\nabla} _{j, {\rm Gr}})\} ). 
$$
\end{corollary}
\begin{proof}
Let $\nabla$ and $\tilde{\nabla}$ denote compatible connections for the two patching 
collections. By the previous lemma, they are both compatible with the common refined patching collection.
By Lemma \ref{nilinv} applied to the refinement, 
$\widehat{c_p}(E,\nabla  ) =  \widehat{c_p}(E,\tilde{\nabla} )$. 
But $\widehat{c_p}(E,\nabla  )$ and $\widehat{c_p}(E,\tilde{\nabla} )$ are respectively 
ways of calculating 
$\widehat{c_p}(X,E,\{ (V_i,W^i, \nabla _{i, {\rm Gr}})\} )$ and 
$\widehat{c_p}(X,E,\{ (\tilde{V}_j,\tilde{W}^j, \tilde{\nabla} _{j, {\rm Gr}})\} )$,
so these last two are equal.  
\end{proof}


\subsection{Construction of a patched connection}


Suppose we have a pre-patching collection $(X,E,\{ (V_i,W^i, \nabla _{i, {\rm Gr}})\} )$. In order to construct a compatible
connection, we need the following compatibility condition on the intersections $V_i\cap V_j$.

\begin{condition}
\label{patchcomp}
We say that the pre-patching collection satisfies the {\em patching compatibility condition} if for any point $x\in V_i\cap V_j$
there is a neighborhood $V'_x$ of $x$ and a common refinement $\tilde{W}^x$ of both filtrations $W^i$ and $W^j$ on $V'_x$,
consisting of bundles coming from these filtrations,
such that the connections $\nabla _{i,{\rm Gr}}$ and  $\nabla _{j,{\rm Gr}}$ both preserve the filtrations induced by
$\tilde{W}^x$ on the respective associated graded bundles $Gr^{W^i}(E|_{V'_x})$ and $Gr^{W^j}(E|_{V'_x})$. Furthermore we
require that the induced connections on $Gr^{\tilde{W}^x}(E|_{V'_x})$ be the same. 
\end{condition}

\begin{lemma}
\label{commonrefinement2}
Suppose  $(X,E,\{ (V_i,W^i, \nabla _{i, {\rm Gr}})\} )$ is a pre-patching collection which
satisfies the criterion \ref{patchcomp}.
Then for any point $x$ lying in several open sets $V_{i_1}, \ldots , V_{i_N}$, there is a smaller neighborhood
$x\in V''_x\subset V_{i_1}\cap \cdots \cap V_{i_N}$ and a common refinement $U^x$ of all of the filtrations $W^{i_j}$, $j=1,\ldots , N$
on $E|_{V''_x}$, such that the induced filtrations on any of the associated graded pieces ${\rm Gr}^{W^{i_j}}(E|_{V''_x})$ are
preserved by the connections $\nabla _{i_j,{\rm Gr}}$, and the connections all induce the same connection on 
the associated graded of the common refined filtration $U^x$. 
\end{lemma}
\begin{proof}
Fix $x\in V_{i_1}\cap \cdots \cap V_{i_N}$.
Choose any neighborhood of $x$ contained in the intersection. The filtrations $W^{i_j}$, $j=1,\ldots , N$ admit pairwise common refinements
by Condition \ref{patchcomp}. Therefore by Corollary \ref{multiplerefine}, they admit a single refinement $U^x$ common to all, and furthermore
the component bundles $U^x_a$ are taken from among the component bundles of the different $W^{i_j}$.

Now, on an associated-graded piece ${\rm Gr}^{W^{i_j}}(E|_{V''_x})$ consider one of the bundles in the induced filtration
${\rm Gr}^{W^{i_j}}(U^x_a)$. This comes from another filtration, so it is equal to some ${\rm Gr}^{W^{i_j}}(W^{i_{\ell}}_b)$. Then Condition \ref{patchcomp}
says that this bundle is preserved by the connection $\nabla _{i_j,{\rm Gr}}$. This shows the next to last phrase.

Finally, choose some associated-graded piece $U^x_a / U^x_{a-1}$, and two other indices $i_j$ and $i_{\ell}$.
There is an index $b$ such that
$$
W^{i_j}_{b-1} \subset U^x_{a-1} \subset U^x_{a} \subset W^{i_j}_{b}.
$$
Similarly there is an index $c$ such that
$$
W^{i_{\ell}}_{c-1} \subset U^x_{a-1} \subset U^x_{a} \subset W^{i_{\ell}}_{c}.
$$
Now $U^x_a / U^x_{a-1}$ is a subquotient of one of the terms $G$ in the associated-graded for the common refinement of
$W^{i_j}$ and $W^{i_{\ell}}$. The connections $\nabla _{i_j,{\rm Gr}}$ and $\nabla _{i_{\ell},{\rm Gr}}$
define the same connection on $G$, and both of them preserve the subbundles of $G$ corresponding to 
$U^x_{a-1}$ and $U^x_{a}$. Hence they induce the same connection on $U^x_a / U^x_{a-1}$. This proves the last phrase.
\end{proof}

\begin{theorem}
\label{patchconstr}
Suppose 
$(X,E,\{ (V_i,W^i, \nabla _{i, {\rm Gr}})\} )$ is a pre-patching collection which satisfies the above patching compatibility
condition \ref{patchcomp}. Then it has a refinement which is a patching collection, that is to say there exists a compatible patched connection
for a refined pre-patching collection. 
\end{theorem}
\begin{proof}
To begin, we can choose over each $V_i$ a connection $\nabla _i$ on $E|_{V_i}$ such that $\nabla _i$ preserves
the filtration $W^i$ and induces the connection $\nabla _{i,{\rm Gr}}$ on the associated-graded. One way to do this
for example is to choose a $\cC ^{\infty}$ hermitian metric on $E$ which induces a splitting 
$$
{\rm Gr}^{W_i}(E|_{V_i})\cong E|_{V_i},
$$
and use this isomorphism to
transport the connection $\nabla _{i,{\rm Gr}}$.

Choose a partition of unity $1= \sum _i \zeta _i$ with ${\rm Supp}(\zeta _i)$ relatively compact 
in $U_i$. Consider the {\em patched connection}
$$
\nabla ^{\#}:= \sum _i \zeta _i \nabla _i .
$$
It is well-defined as a $\cC^{\infty}$ operator $E\rightarrow A^1(E)$ (where $A^{\cdot}$ 
denotes the differential forms on $X$), 
because the $\zeta _i$ are compactly supported in the open set $U_i$ of
definition of $\nabla _i$. Furthermore, it is a connection operator, that is it satisfies Leibniz' rule:
$$
\nabla ^{\#}(a e) = \sum _i \zeta  _i \nabla _i (a e) = \sum _i a \zeta _i \nabla _i (e) + 
( \sum _i \zeta _i d(a))  e = a\nabla ^{\#}(e) + d(a) e
$$
using $\sum _i \zeta _i = 1$. 

We would now like to consider compatibility of $\nabla ^{\#}$ with the filtrations. Choose $x\in X$. 
Let $i_1,\ldots , i_N$ be the indices
for which $x$ is contained in $Supp (\zeta _{i_j})$. Choose a neighborhood $V''_x$ as in the situation 
of Lemma \ref{commonrefinement},
contained in $V_{i_1}\cap \cdots \cap V_{i_N}$ but not meeting the support of any $\zeta _j$ for $j$ 
not in $\{ i_1,\ldots , i_N\}$. 
Let $U^x$ be the common refinement of the filtrations $W^{i_j}$ given by Lemma \ref{commonrefinement}. 

Each of the connections $\nabla _{i_j}$ preserves every $U^x_a$. Indeed, $U^x_a$ is sandwiched between 
$W^{i_j}_{b-1}$ and $W^{i_j}_{b}$,
and $\nabla _{i_j}$ induces the connection $\nabla _{i_j,{\rm Gr}}$ on $W^{i_j}_{b} / W^{i_j}_{b-1}$. 
By hypothesis, and Lemma \ref{commonrefinement},
the connection $\nabla _{i_j,{\rm Gr}}$ preserves the image of $U^x_a$ in $W^{i_j}_{b} / W^{i_j}_{b-1}$, 
therefore $\nabla _{i_j}$
preserves $U^x_a$. 

Furthermore, the connections $\nabla _{i_j}$ all induce the same connection on 
$U^x_a/ U^x_{a-1}$, as follows from the same statement for
the connections $\nabla _{i_j,{\rm Gr}}$ in Lemma \ref{commonrefinement}. 

The neighborhoods $V''_x$ cover $X$. Together with the filtrations $U^x$ and the 
connections induced by
any of the $\nabla _ {i_j}$ on ${\rm Gr}^{U^x}(E|_{V''_x})$ this gives a pre-patching 
collection refining the original one. 

The connection $\nabla ^{\#}$ is compatible with this new pre-patching collection. Indeed, it is a sum of terms $\nabla _i$ and
on any open set $V''_x$ the only terms which come into play are the $\nabla _{i_j}$ which preserve the filtration $U^x$
and induce the given connections on ${\rm Gr}^{U^x}(E|_{V''_x})$. 
This $\nabla ^{\#}$ preserves the filtration $U^x$. 
By the partition of unity condition 
$\sum \zeta _{i_j}= 1$ on
$V'' _x$, the patched connection $\nabla ^{\#}$ induces the given connection on each ${\rm Gr}^{U^x}(E|_{V''_x})$.
\end{proof}


\subsection{The patched connection for a representation unipotent along a smooth divisor}
\label{patchsmoothdiv}


If we have tried to be somewhat general in the previous presentation, 
we only use the construction of the patched connection
in the simplest case. 
Suppose $X$ is a smooth variety and $D\subset X$ is a closed smooth irreducible divisor. Choose the basepoint 
$x\in X-D$ and suppose we have a representation $\rho :  \pi _1(X-D,x)\rightarrow GL_r(\comx )$. 

Let $\gamma$ denote the path going from $x$ out
to a point near $D$, once around counterclockwise, then back to $x$.  
We assume that $\rho$ is {\em unipotent at infinity}, that is to say that the $\rho (\gamma )$ is a unipotent matrix. 

As usual, fix the following two neighborhoods covering $X$. First, $U := X-D$ is the complement of $D$. Then $B$ is a tubular neighborhood of $D$.
Let $B^{\ast}:= U \cap B$, it is the complement of $D$ in $B$ 
otherwise known as the {\em punctured tubular neighborhood}. We have a projection $B\rightarrow D$, making $B$ into a 
disc bundle and $B^{\ast}$ into a punctured-disk bundle over $D$. In terms of the previous notations, the index set is $I= \{ 0,1\}$
and $U_0=U, \; U_1 = B$ with $U_0\cap U_1 = B^{\ast}$.

Let $\ov{E}$ denote the holomorphic vector bundle on $X$ which is the Deligne canonical 
extension of the flat bundle associated to $\rho$.
Let $\nabla$ denote the flat connection on $E$. 
In particualar, $(E,\nabla )$ is the flat bundle over $U$ associated to $\rho$.

Fix the trivial filtration $W^0_0:= E$ and $W^0_{-1}=0$ over the open set $U=U_0$. 
The assocated-graded is the whole bundle $E$ and we take $\nabla _{0,{\rm Gr}}:= \nabla $. 

Recall that a {\em graded-extendable filtration} on $E|_{B^{\ast}}$ is a filtration 
$\{ W_k\}$ by strict $\nabla$-flat subbundles, such that the induced connection $\nabla _{Gr}$ on
$Gr^W(E|_{B^{\ast}})$ extends to a connection over $B$. Note that the $W_k$ extend to strict
subbundles of $\ov{E}|_B$, indeed we take the canonical extension of $W_k$ with respect to the connection
induced by $\nabla$. Hence we are given natural bundles $Gr^W(\ov{E}|_{B})$ and the graded-extendability
condition says that the connection induced by $\nabla$ on these graded bundles, should be nonsingular along $D$.

Examples of such filtrations include the kernel filtration (see Lemma \ref{tube}) 
or the monodromy weight filtration along
$D$, using the hypothesis that $\rho$ is unipotent at infinity. 

On $U_1=B$ let $W^1= \{ W_k\}$ denote some choice of graded-extendable filtration. 

Let $\nabla _{1,{\rm Gr}}$ be the connection induced by $\nabla$ over $B^{\ast}$, 
projected to the associated-graded ${\rm Gr}^{W^i}(E|_{B^{\ast}})$
and then extended from $B^{\ast}$ to a connection on $\ov{E}|_B$,
well-defined over all of $B$. 

The resulting collection of neighborhoods, filtrations and connections on the 
associated-graded's, is a pre-patching collection on $X$. 
 
\begin{lemma}
Suppose $\rho$ is a representation of $\pi _1(U)$ which is unipotent at infinity,
and choose a graded-extendable filtration $W$ on the corresponding flat bundle restricted
to $B^{\ast}$.
The pre-patching collection associated to $(\rho , W)$ by the above discussion
satisfies the compatibility condition \ref{patchcomp}, 
hence by Theorem \ref{patchconstr} it admits a compatible patched connection
denoted $\nabla ^{\#}$. 
\end{lemma}
\begin{proof}
This is obvious, since the filtration on $U$ is the trivial filtration so over the intersection
$B^{\ast}$ it clearly admits a common
refinement with the filtration $\{ W_k\}$ on $B$. 
\end{proof}

Since there are only two open sets and a single intersection, it is easy to write down
explicitly the patched connection $\nabla ^{\#}$ here. Furthermore, there is no
need to refine the pre-patching collection in this case.  

The partition of unity consists of a single function
$\zeta$  supported on $B$ with $1-\zeta$ supported on $U$. We choose a $\cC^{\infty}$
trivialization of the filtration over $B$, $\ov{E}|_B \cong Gr^W(\ov{E}|_B)$. 
Thus $\nabla _{1,{\rm Gr}}= Gr^W(\nabla )$ gets transported to a connection $\nabla _B$ on
$\ov{E}|_B$. Then
$$
\nabla ^{\#} = (1-\zeta )\nabla + \zeta \nabla _B
$$
is a $\cC^{\infty}$ connection on $\ov{E}$ over $X$. Over $B$ it preserves the filtration $W$
and on $Gr^W(\ov{E}|_B)$ it induces the given connection $\nabla _{1,{\rm Gr}}$ which is flat.
Thus, $\nabla ^{\#}$ is locally nil-flat in the easy sense that, over the open set $U'\subset U$ which is
the complement of the support of $\zeta$, it is flat (equal to the original $\nabla$), while
over $B$ it is upper triangular with strictly upper triangular curvature, with respect to the filtration
$W$. We have $X=U'\cup B$.

\begin{corollary}\label{patchedclass}
We obtain secondary classes 
$$
\widehat{c}_p(\rho , W):= \widehat{c}_p(\nabla ^{\#})\in H^{2p-1}(X, \comx / \Z ) 
$$
from the patched connection.
These are independent of the choices of neighborhoods and partitions of unity used to define $\nabla ^{\#}$.
\end{corollary} 
\begin{proof}
It follows directly from Proposition \ref{nildef} and Lemma \ref{nilinv}: the Chern forms of 
$\nabla ^{\#}$ vanish identically everywhere, because the curvature is everywhere strictly upper
triangular in some frame. Thus, the Cheeger-Simons class in differential characters,
lies in the subgroup $H^{2p-1}(X, \comx / \Z )$. This argument was mentioned in Corollary 2.4 of
Cheeger-Simons \cite{Ch-Si}. Independence of choices follows from 
Lemma \ref{nilinv}. 
\end{proof}

Using the extension of Deligne-Sullivan \cite{De-Su} given by Proposition \ref{trivext}, eventually after going to a finite
cover of $X$ ramified only over $D$, we can assume that the canonical extension $\ov{E}$ is trivial.
Thus, we can apply the variational way of understanding the Chern-Simons class of $\nabla ^{\#}$ in differential 
characters.

We will see in \S \ref{sec-rigidity} below (Corollary \ref{indepfilt}) 
that the class $\widehat{c}_p(\rho , W)$ is also independent of the choice of 
graded-extendable filtration $W$, so it can also be denoted $\widehat{c}_p(\rho /X)$.

In the more general normal-crossings case, one would like to apply the general considerations of the previous subsections
to obtain a construction. However, we found that it is not immediately obvious how to produce a covering and appropriate filtrations
such that the filtrations admit a common refinement on the intersections (Condition \ref{patchcomp}). The structure of the commuting nilpotent
logarithms of monodromy transformations is complicated. Some structure results are known, for example the monodromy
weight filtrations of $\sum a_i N_i$ are the same whenever $a_i>0$, a result which is now generalized from the case of variations of Hodge structure
to any harmonic bundle by Mochizuki \cite{Mochizuki}. However, this doesn't provide an immediate answer for patching the connection. 
This is one of the main reasons
why, in the present paper, we are treating the case of a smooth divisor only. 

See also Remark \ref{unclear-def} below for a somewhat different difficulty in the normal crossings case.

\section{Compatibility with the Deligne Chern class}
\label{delignecompatibility}


Suppose $X$ is a smooth complex projective variety. 
Consider the following situation: $E$ is a holomorphic vector bundle on $X$ with
holomorphic structure operator $\delbar$. Suppose $\nabla _1$ is a connection obtained by
the patching construction. We assume that in a standard neighborhood $V_x$ of any point 
$x$, the local filtrations $W^x$ are by
holomorphic subbundles of $(E,\delbar )$, and that the holomorphic structure on the graded pieces
$Gr^{W_x}_k(E)$ coincides with the $(0,1)$ part of the flat connections induced by $\nabla _1$
on these pieces. In this situation, we claim that the Chern-Simons class in $H^{2p-1}(X,\comx / \Z )$
defined by the patched connection $\nabla _1$, projects to the Deligne Chern class of
$(E,\delbar )$ in $H^{2p}_{\cD}(X, \Z (p))$. 

For this we use the formalism of $F^1$-connections introduced by Dupont, Hain and Zucker.
Their method fully works only when $X$ is compact. 
Recall that this is a variant of the differential character construction. Let $DHZ^{k,k+1}$ denote
the group of analogues of differential characters used by Dupont, Hain and Zucker. We have
an exact sequence, by quotienting the exact sequence in \eqref{diffeqn} by the Hodge piece
$F^p$;
$$
0\rightarrow H^{2p}_{\cD}(X, \Z (p)) \rightarrow DHZ^{2p - 1, 2p} \rightarrow 
\frac{A^{2p}_{\rm cl}(X, \Z )}{(A^{p,p} + \cdots + A^{2p, 0})\cap A^{2p}_{\rm cl}(X, \Z )} 
\rightarrow 0.
$$
Here $ DHZ^{2p - 1, 2p}:= \widehat{H^{2p}}(X,\comx/\Z)/F^p$.

Suppose we have a connection $\nabla _0$ compatible
with $\delbar$; this means that $\nabla _0 ^{0,1}=\delbar $. In
\cite{DHZ}, it is shown that the differential character defined by the connection $\nabla _0$ projects from 
$\what{H^{2p}}(X,\comx/\Z)$ to $DHZ^{2p - 1, 2p}$, to a class which goes to zero in the 
term ``closed forms modulo the Hodge filtration'' on the right, and which thus comes from a class in the 
Deligne cohomology on the left; and that this is the same as the Deligne class of $(E,\delbar )$. 

In our case, we construct $\nabla _0$ as follows: take 
$$
\nabla _0^{0,1}:=\delbar , \;\;\; \nabla _0 ^{1,0} := \nabla _1^{1,0}.
$$
This is by definition compatible with $\delbar$, so its class in $\what{H^{2p}}(X,\comx/\Z)$ projects to 
the Deligne Chern class by Dupont-Hain-Zucker \cite{DHZ}. 

On the other hand, notice that $\nabla _0$ defines a connection which preserves the filtration $W^x$
on the neighborhood of any $x\in X$, and which induces the original flat connection on the associated graded pieces.
Preserving the filtration is because $\delbar $ and $\nabla _1$ both preserve the filtration. On the graded pieces,
recall that $\nabla _1$ induces the flat connection, and also the flat connection has the same operator $\delbar$ as comes from 
$E$. In particular $\nabla ^{0,1}_0 = \nabla _1^{0,1}$ on the graded pieces, so $\nabla _0$ induces the same connection as $\nabla _1$ here.

 From this we get that $\nabla _0$ also has strictly upper triangular curvature form $\Omega _0$, so its class in 
$\what{H^{2p}}(X,\comx/\Z)$ projects to zero in $A^{2p}_{\rm cl}(X, \Z )$. Thus, $\nabla _0$ defines a class in 
$H^{2p-1}(X, \comx / \Z )$. This class projects to the Deligne Chern class, by the result of \cite{DHZ}.

To finish the proof of compatibility, we will show that $\nabla _0$ and $\nabla _1$ define the same class in
$H^{2p-1}(X, \comx / \Z )$. 

\begin{lemma}\label{lift}
The Chern-Simon classes $\what{c_p}(E,\nabla_0)$ and $\what{c_p}(E,\nabla_1)$ are equal.
\end{lemma}
\begin{proof}
For this, connect the connection $\nabla _1$ to $\nabla _0$ by an affine path of connections
$$
\nabla _t = t\nabla _1 + (1-t)\nabla _0.
$$
For $t=0,1$ this coincides with the previous ones. Let $\Omega _t$ denote the curvature form of $\nabla _t$ and
let $\nabla ' _t$ denote the derivative with respect to $t$. 
The rest of the proof is the same as in Lemma \ref{nilinv}.
\end{proof}

Denote this class by $\what{c_p}(E,\nabla)$, for $p\geq 1$.
We have thus shown, together with Lemma \ref{lift} and Corollary \ref{patchedclass}:

\begin{proposition}\label{gluedsecondary}
Suppose $X$ is a smooth complex projective variety, with 
a logarithmic connection $(E,\nabla)$ on $X$ with nilpotent 
residues along a smooth and irreducible divisor $D$. It restricts to a flat connection $(E_U,\nabla_U)$ on the complement 
$U:=X-D$. Let $B$ be a tubular neighbourhood 
of the divisor $D$ as obtained in Lemma \ref{tube} and $({E}_{B},\nabla_{B})$ be the restriction of $(E,\nabla)$ on $B$.
Then the secondary classes $\what{c_i}({E}_{B},\nabla_{B})$ and $\what{c_i}(E_U,\nabla_U)$ glue together to give a canonically determined class
$\what{c_i}(E,\nabla) \in H^{2i-1}(X,\comx/\Z)$, for $i\geq 1$. 
The classes
$\what{c_i}(E,\nabla)$ lift the Deligne Chern classes $c_i^\cD({E})$ under the projection
$$
H^{2i-1}(X,\comx/\Z) \lrar H^{2i}_\cD({X},\Z(i)).
$$ 
\end{proposition}


\section{Rigidity of the secondary classes}
\label{sec-rigidity}


In this section we would like to show that the secondary classes are invariant under deformations of the representation.
In the flat case this is a consequence of a well-known formula for the variation of the secondary class. 
In our case we need to be somewhat careful about the local filtrations. 

Before getting to the rigidity result we look at the construction from a somewhat more general point of view.
We are given an open covering of $X$ by neighborhoods $U$ and $B$. 
In our situation $B$ and even $B^{\ast}:= B\cap U$ are connected, and 
$\pi _1(B^{\ast})\rightarrow \pi _1(B)$ is surjective.

We have a representation $\rho$ of $\pi _1(U)$,
corresponding to a flat bundle $(E,\nabla )$ and to a local system $L= E^{\nabla}$ on $U$. Denote by 
$L_{B^{\ast}}$ the restriction
of $L$ to a local system on $B^{\ast}$. 
Suppose we are given a filtration
$W$ of $L_{B^{\ast}}$  such that the graded pieces $Gr^{W}_k$ extend to local systems over $B$.

The
patching construction with trivial filtration over $U$ gives
a patched connection and a secondary class which we denote here by
$$
\widehat{c}_p(E,\nabla , W ) \in H^{2p-1}(X, \comx / \Z )
$$
to emphasize dependence on the filtration.

Recall from Corollary \ref{commonrefinement}: if $\tilde{W}$ is a 
different filtration such that $W$ and $\tilde{W}$ admit a common refinement, then
$$
\widehat{c}_p(E,\nabla , W) = \widehat{c}_p(E,\nabla , \tilde{W} ).
$$

\begin{lemma}
\label{filtstring}
Suppose $W$ and $\tilde{W}$ are two filtrations of $L_{B^{\ast}}$ by
sub-local systems, such that the associated graded pieces extend as local systems on $B$. Then these are connected
by a string of  filtrations 
$$
W(0)=W, W(1), \ldots , W(a_1)= \tilde{W}
$$
such that $W(a)$ satisfies the same conditions for any $0\leq a \leq a_1$: it is a filtration of
$L_{B^{\ast}}$ by sub-local systems, 
such that the associated graded pieces extend as local systems on $B$. Furthermore, any adjacent ones
$W(a-1)$ and $W(a)$ admit a common refinement for $0< a \leq a_1$. 
\end{lemma}
\begin{proof}
The proof is by induction on the rank $r$ of $E$. It is easy for $r=1$, so we assume $r>1$ and that it is known for
representations of rank $r' <r$. 

Recall that we have the {\em canonical Jordan-H\"older filtration} $W^{\rm JH}$ of $L_i$. The first step
$W^{\rm JH}_0$ is {\em socle} or largest semisimple subobject of $L_{B^{\ast}}$, and
the remainder of the filtration is determined inductively by the condition that it should induce
the canonical Jordan-H\"older filtration of $L_{B^{\ast}}/W^{\rm JH}_0$. We note that this filtration 
has the property that
the associated graded pieces extend as local systems on $B$. Indeed, the associated-graded of 
$W(JH)$ is the semisimplification of $L_{B^{\ast}}$, but since $L_{B^{\ast}}$ has at least one filtration $W$ whose
graded pieces extend to $B$, it follows that the pieces of the semisimplification all extend to $B$. In this
argument we are using surjectivity of $\pi _1(B^{\ast})\rightarrow \pi _1(B)$ so that an extension to $B$
is unique if it exists.  

Denote by
$W_{b}$ the first nontrivial piece in the filtration $W$. Then the socle of $W_b$ is a 
nontrivial sub-local system $V\subset L_{B^{\ast}}$,
so by the universal property of the socle of $L_{B^{\ast}}$ 
this is contained in $W^{\rm JH}_0$. Now $V$ is a subsystem of the first elements
of both filtrations $W$ and $W^{\rm JH}$. Hence, $W$ and $W^{\rm JH}$ induce filtrations on 
$L_{B^{\ast}}/V$. By the inductive hypothesis,
these two filtrations are connected by a sequence as in the conclusion of the lemma. 
Lifting and including $V$ as the first element,
we obtain a sequence of filtrations connecting $W \cup \{ V \}$ to $W^{\rm JH}\cup \{ V\}$. We can then add on 
$W$ and $W^{\rm JH}$ to
the ends of this sequence, so we obtain a sequence connecting $W$ to $W^{\rm JH}$. Similarly there is a sequence connecting
$\tilde{W}$ to $W^{\rm JH}$. Putting them together we obtain a sequence 
connecting $W$ to $\tilde{W}$. This completes the proof.
\end{proof}

\begin{corollary}
\label{indepfilt}
Suppose $W$ and $\tilde{W}$ are two filtrations of $L_{B^{\ast}}$ by
sub-local systems, such that the associated graded pieces extend as local systems on $B$. Then 
$$
\widehat{c}_p(E,\nabla , W ) = \widehat{c}_p(E,\nabla , \tilde{W} ).
$$
\end{corollary}
\begin{proof}
Use the sequence of filtrations $W(a)$ constructed in the previous lemma. 
The secondary classes of adjacent elements are the same:
$$
\widehat{c}_p(E,\nabla , W(a-1) ) = \widehat{c}_p(E,\nabla ,W(a) ),
$$
because the adjacent elements admit a common refinement. Therefore 
$$
\widehat{c}_p(E,\nabla , W ) = \widehat{c}_p(E,\nabla , W(0) ) = \widehat{c}_p(E,\nabla , W(a_1) )
=\widehat{c}_p(E,\nabla , \tilde{W} ).
$$
\end{proof}

This corollary says that, while we used the monodromy weight filtrations as a canonical way of 
defining the secondary classes,
we could have used any filtrations compatible with the flat connection and having associated-graded 
which extend as flat bundles on
$B$. In view of this corollary, we now denote the secondary classes by
$$
\widehat{c}_p(\rho /X) \in H^{2p-1}(X, \comx / \Z ).
$$

\begin{corollary}
\label{functoriality}
These classes are additive in $\rho$  and contravariantly functorial in $(X,D)$.
\end{corollary}
\begin{proof}
Suppose $\rho _1$ and $\rho _2$ are representations on $U=X-D$ unipotent around $D$. Choose
graded-extendable filtrations $W^1$ for $\rho _1|_{B^{\ast}}$ and  $W^2$ for $\rho _2|_{B^{\ast}}$.
Then $(W^1\oplus W^2)_i:= W^1_i \oplus W^2_i$ is a graded-extendable filtration for $\rho _1\oplus \rho _2$.
>From our construction of patched connections $\nabla ^{\#}_1$ for $(\rho _1,W^1)$  and 
$\nabla ^{\#}_2$ for $(\rho _2,W^2)$, we get a patched connection $\nabla ^{\#}_1
\oplus \nabla ^{\#}_2$ for $(\rho _1\oplus \rho _2,W^1\oplus W^2)$. 
The associated differential character is the sum, because the differential characters are additive on
direct sums of connections with vanishing Chern forms---the terms of the form $Tr(\ldots )Tr(\ldots )\ldots $
don't contribute in the variational formula \eqref{varform}. 
Thus
$$
\widehat{c}_p((\rho _1\oplus \rho _2)/X) =  \widehat{c}_p(\rho _1/X) \widehat{c}_p(\rho _2/X) .
$$
Suppose $f:(X',D')\rightarrow (X,D)$ is a morphism of smooth quasiprojective varieties with smooth
divisors inducing a map $f:X'-D'\rightarrow X-D$, 
and suppose $\rho$ is a representation of $\pi _1(X-D)$ unipotent along $D$.
Then $f^{\ast}$ is unipotent along $D'$. We can choose a tubular neighborhood $B'$ of $D'$
mapping into the tubular neighborhood $B$ of $D$. If $W$ is a graded-compatible filtration for
$\rho$ over $B^{\ast}$ then $f^{\ast}$ is a graded-compatible filtration for $f^{\ast}(\rho )$,
and again by the formula for the patched connection $\nabla ^{\ast}$ for $(\rho , W)$,
we get that $f^{\ast}\nabla ^{\#}$ is a patched connection for 
$(f^{\ast}(\rho ),f^{\ast}W)$. Hence 
$$
\widehat{c}_p(f^{\ast}(\rho )/X') = f^{\ast}\widehat{c}_p(\rho /X) .
$$
\end{proof}

Turn now to the question of rigidity: if we deform the representation then the secondary classes stay the same. 

\begin{lemma}
\label{rigwithfilt}
Suppose we are given a $\cC^{\infty}$ family of representations $\rho (t)$ of $\pi _1(U)$, for $t\in [0,1]$. 
Suppose we are given a 
$\cC^{\infty}$ family of filtrations $W(t)$ by sub-local systems of
$L_{B^{\ast}}(t)$, having the property that the associated-graded
pieces extend across $B$. Then the secondary classes are constant:
$$
\widehat{c}_p(\rho (t)/X ) = \widehat{c}_p(\rho (t')/X), \;\;\; t,t'\in [0,1].
$$
\end{lemma}
\begin{proof}
We may localize in $t$ to smaller intervals if necessary. 
Let $E$ be the $\cC^{\infty}$-bundle underlying the canonical extension of $\rho (0)$.
Then we may identify the bundles corresponding to $\rho (t)$ with $E$, in such a way 
that the filtrations all correspond to the same
filtration by strict subbundles. This is because the elements of the filtrations have 
the same ranks for all $t$, and as $t$ varies
we can redress the subbundles back onto the same original one by a Gramm-Schmidt 
process which is locally unique. 
Here we might cut the interval up into smaller pieces, but still
a finite number by a compactness argument.

Now, $\rho (t)$ corresponds to a connection $\nabla (t)$ on $E|_{U}$, and
$E|_{B}$ has a filtration by strict subbundles $W$ which corresponds to the filtration of
local systems $W(t)$ for each $t$. We obtain a $\cC^{\infty}$ family of patched connections
$\nabla ^{\#}(t)$. 
These are all locally nil-flat, with respect to a constant filtration on each open set. 

Now we apply the usual proof of rigidity of Chern-Simons classes, see
Cheeger-Simons \cite[Proposition 2.9]{Ch-Si}.
They show that the difference between the Chern-Simons classes is given by 
$$
\widehat{c_p}(\nabla ^{\#} (1)) - \widehat{c_p}(\nabla ^{\#} (0)) =
i.\int _0^1 P(\frac{d}{dt}\nabla^{\#}(t)\wedge \Omega^{p-1}_t)dt |_{Z_{2i-1}}
$$
Here $P$ is the trace form defining the $i$-th Chern form and the integral is taken with endpoints $0$ and $1$.
In our case, 
$\frac{d}{dt}\nabla^{\#}(t)$ are upper-triangular and $\Omega ^{p-1}_t$ are strictly upper-triangular. 
Hence, as long as $p-1 > 0$ the trace vanishes.
\end{proof}

Finally, we obtain the rigidity in general.

\begin{proposition}
\label{rigidity}
Suppose $X$ is covered by a Zariski open dense subset $U$ and a tubular neighborhood
$B$ of the irreducible smooth  divisor $D$, 
with $B^{\ast}:= B\cap U$.
Suppose we are given a  continuous family of representations $\rho (t)$ of $\pi _1(U)$ for $t\in [0,1]$,
whose monodromy transformations around the divisor  $\rho (t)(\gamma )$ are unipotent. 
Then the secondary classes are constant:
$$
\widehat{c}_p(\rho (0) /X) = \widehat{c}_p(\rho (1)/X).
$$
\end{proposition}
\begin{proof}
There is an affine modular variety for representations $\rho$ of $\pi _1(U)$ such that 
$\rho (\gamma )$ has trivial characteristic polynomial. In view of this, we
may replace our continuous family of representations by a piecewise algebraic family. 
Then, the interval of definition can be divided up into sub-intervals of the form $[a_i,a_{i+1}]$
such that for $t\in (a_i,a_{i+1})$ the monodromy weight filtrations (or more easily, the kernel filtrations)
of the $N(t):=\log \rho (t)(\gamma )$ vary in an algebraic manner with the same ranks. 
Then at the endpoints of these intervals, the limits of these filtrations are again filtrations
by sub-local systems, such that the associated graded pieces extend across $B$. Note however that
the limiting filtrations will not in general be the monodromy weight filtrations or kernel filtrations
of the $N(t)$. Lemma \ref{indepfilt} shows that the secondary classes defined with these limiting
filtrations, are the same as those defined by the monodromy weight filtrations. 
Then Lemma \ref{rigwithfilt} applies to give
$$
\widehat{c}_p(\rho (a_i) /X) = \widehat{c}_p(\rho (a_{i+1})/X).
$$
Putting these all together, from the first and last intervals we get the statement of the proposition.
\end{proof}

\section{A deformational variant of the patching construction in $K$-theory}
\label{kclassdef-sec}

Reznikov's proof involved an aspect of arithmetical $K$-theory. Starting with a representation $\rho$ defined over a field $F$,
he considered all embeddings $F\hookrightarrow \comx$; for each embedding the volume piece of the regulator vanished, by 
differential geometric arguments. Using Borel's calculation of the rational $K$-theory of $F$,
this then implied that the classifying map $X\rightarrow BSL(F)$ was trivial on the rational homology.
In order to replicate this part of the proof here, we need a version of the regulator construction which is
related to $K$-theory. 

Recall that, with our usual notations, the diagram
$$
\begin{array}{ccc} 
B^{\ast}  & \rightarrow & B \\ 
\downarrow & & \downarrow \\ 
U & \rightarrow & X
\end{array} 
$$
is a homotopy pushout. Indeed, if we were to replace $U$ by the complement $U_0$ of the interior of $B$,
consider $B_0$ a smaller closed tubular neighborhood,  and let $B^{\ast} _0$ be the closure of 
$X-U_0 - B_0$ then $B^{\ast} _0$ is a cylinder on $S$ which is the boundary of $U_0$ or $B_0$
(they can be identified by an isomorphism). Then $X$ is exactly the standard cylindrical construction
of the homotopy pushout for the maps $S\rightarrow B_0$ and $S\rightarrow U_0$, and the diagram
$$
\begin{array}{ccc} 
B^{\ast}_0 = S\times [0,1]  & \rightarrow & B_0 \\ 
\downarrow & & \downarrow \\ 
U_0 & \rightarrow & X
\end{array} 
$$
is homotopy equivalent to the previous one. 

Suppose we have a representation $\rho$ defined over a field $F$, that is
$\rho : \pi _1(U)\rightarrow GL_r(F)$. Because of the homotopy pushout square, the 
problem of constructing a map $X\rightarrow BGL(F)^+$  reduced to giving maps on $B$ and $U$,
plus a comparison over $B^{\ast}$. The representation $\rho$ gives a map $U\rightarrow BGL(F)$.
With the assumption that $\rho$ is unipotent at infinity, choice of a compatible filtration $W$
for $\rho |_{B^{\ast}}:= \rho |_{\pi _1(B^{\ast})}$ induces a representation which extends over $B$ to
$$
Gr^W(\rho |_{B^{\ast}}) : \pi _1(B)\rightarrow GL_r(F)
$$
again giving a map $B\rightarrow BGL(F)$. 

Deformation from a representation to its associated-graded, is a polynomial deformation. 
Thus we can get a map $B^{\ast}\rightarrow BGL(F[t])$ linking the maps on $U$ and $B$.
The deformation theorem in $K$-theory allows us to interpret this as a glueing datum
giving rise to a map $X\rightarrow BGL(F)^+$. 

\subsection{The deformation theorem}

Consider the following situation. Let $F$ be a field. We get two morphisms $e_0, e_1 : F[t]\rightarrow F$  
consisting of evaluation of polynomials at $0$ and $1$ respectively. Inclusion of constants is $c:F\rightarrow F[t]$, 
whose composition with $e_i$ is the identity.  
The {\em deformation theorem} in $K$-theory (Quillen \cite{Quillen}, see \cite{Srinivas} or \cite{Rosenberg}) 
says that all of these maps induce homotopy equivalences 
of $K$-theory spaces  
$$ 
BGL(F) ^+ \rightarrow BGL(F[t]) ^+ \stackrel{e_0 \; {\rm or} \; e_1}{\longrightarrow} BGL(F) ^+ . 
$$

Define a space $BGL(F) ^+ _{\rm def}$ to be the homotopy pushout in the diagram  
$$ 
\begin{array}{ccc} 
\rightarrow BGL(F[t]) ^+ & \rightarrow & BGL(F) ^+ \\ 
\downarrow & & \downarrow \\ 
BGL(F) ^+ & \rightarrow & BGL(F) ^+_{\rm def}  
\end{array} . 
$$ 
Explicitly, $BGL(F) ^+_{\rm def}$ is obtained by glueing the cylinder 
$BGL(F[t]) ^+ \times [0,1]$ to two copies of $BGL(F) ^+$ along the evaluation maps 
$$
e_0 : BGL(F[t]) ^+ \times \{ 0 \} \rightarrow BGL(F) ^+
$$
and 
$$
e_1 : BGL(F[t]) ^+ \times \{ 1 \} \rightarrow BGL(F) ^+ .
$$
We can express $BGL(F) ^+_{\rm def}$ as a union of two open sets: the first is the glueing of 
$$
BGL(F[t]) ^+ \times [0,1) \;\;\mbox{to a copy of}\;\; BGL(F) ^+\;\;\mbox{by}\;\; e_0, 
$$
the second is the
glueing of 
$$
BGL(F[t]) ^+ \times (0,1) \;\;\mbox{to the other copy of}\;\; BGL(F) ^+\;\;\mbox{by}\;\; e_1. 
$$
And their
intersection is $BGL(F[t]) ^+ \times (0,1)$. Each of the open sets retracts to $BGL(F)^+$. 
The deformation theorem implies that  the top and left maps in the above square are homotopy equivalences.
The Van Kampen theorem plus Mayer-Vietoris implies that the maps $BGL(F) ^+\rightarrow BGL(F)_{\rm def}^+$
are homotopy equivalences.

\begin{corollary} 
\label{def-isom} 
There is an isomorphism on cohomology with any coefficients  
$$ 
H^{\ast} (BGL(F) ^+_{\rm def} , k) \cong H^{\ast} (BGL(F) ^+ , k) = H^{\ast} (BGL(F) , k) . 
$$ 
\end{corollary} 

\subsection{Deformation patching}
 
Now we say that a {\em deformation patching datum} for our diagram of pointed connected spaces  
$$ 
(U,p)\stackrel{a}{\leftarrow} (B^{\ast}, p) \stackrel{b}{\rightarrow} (B, p) 
$$ 
is a triple of representations  
$$ 
\eta _U : \pi _1(U, p) \rightarrow GL(F), 
$$ 
$$ 
\eta _B : \pi _1(B, p) \rightarrow GL(F), 
$$ 
and 
$$ 
\eta _{B^{\ast}} : \pi _1(B^{\ast},p) \rightarrow GL(F[t]) 
$$ 
such that  
$$ 
e_1\circ \eta _{B^{\ast}}  = a^{\ast}(\eta _U) \;\;\mbox{and} \;\;  
e_0\circ \eta _{B^{\ast}}  = b^{\ast}(\eta _B) . 
$$ 
In more geometric terms, we require representations on $U$ and $B$, plus a deformation between their 
restrictions to $B^{\ast}$.  Typically, the representations will go into a finite-dimensional subgroup
of the form $GL_r(F)$ (resp. $GL_r(F[t])$). 
 
Suppose $X$ is the homotopy pushout of $U\leftarrow B^{\ast} \rightarrow B$.  
We can assume for example that, as in the geometric situation, this diagram is homotopic
to $U_0 \leftarrow S \leftarrow B_0$ and $B^{\ast}$ is a cylinder $S\times (0,1)$,
furthermore $U = U_0\cup B^{\ast}$ and $B = B_0 \cup B^{\ast}$ retract to $U_0$ and $B_0$ respectively. 
 
Given a deformation patching datum $(\eta _U, \eta _B, \eta _{B^{\ast}})$, 
the representations give maps 
$$
U,B \rightarrow BGL(F)\;\; \mbox{and}\;\; B^{\ast} \rightarrow BGL(F[t]).
$$  
By functoriality of homotopy pushout, this gives a homotopy class of maps 
$$ 
X \rightarrow BGL(F) ^+_{\rm def}, 
$$ 
in particular using Corollary \ref{def-isom} we get a map 
$$ 
H^{\ast}(BGL(F), k) \rightarrow H^{\ast}(X,k). 
$$ 
 
If $F\subset \comx$ then we can apply this to the universal regulator class in $H^{2p-1}(BGL(F), \comx / \Z )$ 
to get a {\em deformation regulator class} denoted 
$$
\widehat{c}_p^{\rm def}(\eta _U, \eta _B, \eta _{B^{\ast}})\in H^{2p-1}(X, \comx / \Z ).
$$ 
Its imaginary part will be called the {\em deformation volume regulator} denoted 
$$
Vol_{2p-1}^{\rm def}(\eta _U, \eta _B, \eta _{B^{\ast}})\in H^{2p-1}(X, \R ).
$$ 

\subsection{The deformation associated to a filtration}\label{deformation-filtration}

We now point out that we get a deformation patching in our standard canonical-extension
situation. 
Assume here that the map $\pi _1(B^{\ast})\rightarrow \pi _1(B)$ is surjective,
as is the case if $B$ is the tubular neighborhood of a divisor $D$ and $B^{\ast}=B-D$. 

Say that a {\em filtration patching datum} consists of a representation 
$$
\rho : \pi _1(U)\rightarrow GL(V)
$$
for a finite dimensional vector space $V$, plus a filtration $W$ of $V$ such that
$W$ is invariant under the action of $\pi _1(B^{\ast})$, and the induced action
of $\pi _1(B^{\ast})$ on $Gr^W(V)$ factors through a representation
$$
Gr^W(\rho /B): \pi _1(B)\rightarrow GL(Gr^W(V)).
$$
Note that this factorization is unique because of the assumption that 
$\pi _1(B^{\ast})\rightarrow \pi _1(B)$ is surjective.  In the divisor situation,
such a filtration will exist if and only if $\rho (\gamma )$ is unipotent. 

Given a filtration patching datum $(\rho , W)$  
we can define a deformation patching datum as follows. Choose a splitting for the 
filtration $V= \bigoplus _iV_i$ which yields $V\cong Gr^W(V)$, 
and furthermore choose a compatible basis for $V$ which gives 
$$
GL(V)\cong GL_r(F)\hookrightarrow GL(F). 
$$
Composing with $\rho$ gives $\eta _U: \pi _1(U)\rightarrow GL(F)$. 
On the other hand, composing with the representation $Gr^W(\rho /B)$ gives 
$$
\eta _B : \pi _1(B)\stackrel{Gr^W(\rho /B)}{\lrar} GL(Gr^W(V))\cong GL(V)\cong GL_r(F) \hookrightarrow GL(F). 
$$
For $\eta _{B^{\ast}}$, notice that the matrices preserving $W$ are 
block upper triangular. Then define a deformation between $\rho |_{\pi _1(B^{\ast})}$ and 
$Gr^W(\rho /B) |_{\pi _1(B^{\ast})}$ as follows (this is the same deformation as was refered to 
in the proof of Lemma \ref{model} communicated by Deligne \cite{De3}). Using the decomposition of $V$ we get  
$$ 
End(V)\cong \bigoplus _{i,j} Hom (V_i, V_j). 
$$ 
Let $End_W(V)$ be the subspace of endomorphisms preserving $W$, so  
$$ 
End_W(V)\cong \bigoplus _{i\geq j} Hom (V_i, V_j). 
$$ 
Define a map $\psi _t : End _W(V)\rightarrow End_W(V)$ by multiplying by $t^{j-i}$ on the piece $Hom (V_i,V_j)$. 
At $t=1$ this is the identity and at $t=0$ this is the projection to the block diagonal pieces. 
Note that $\psi _t (MM') = \psi _t(M)\psi _t(M')$. Thus if $p:\Gamma \rightarrow End_W(V)$ is a group representation 
(whose image lies in the subset of invertible matrices) then the function $\psi _t\circ p$ is again  
a group representation. This gives a deformation $\Gamma \rightarrow GL_r(F[t])$ whose value at $0$ is 
the original $p$ and whose value at $0$ is the associated-graded of $p$. Apply this to  
the restriction $\rho  |_{\pi _1(B^{\ast})}$ with $\Gamma = \pi _1(B^{\ast})$. 
This gives a representation  
$$ 
\eta ' : \pi _1(B^{\ast})\rightarrow GL_r(F[t]) 
$$ 
such that $e_1\eta '$ is the representation $\rho |_{\pi _1(B^{\ast})}$, and whose value 
$e_0\eta '$ is the associated-graded, which is equal to $Gr^W(\rho /B) |_{\pi _1(B^{\ast})}$.
Letting $\eta _{B^{\ast}}$ be the composition of $\eta $ with the inclusion $GL_r(F)\hookrightarrow GL(F)$ we 
have completed our deformation patching datum $(\eta _U, \eta _B, \eta _{B^{\ast}})$ 
associated to the filtration patching datum $(\rho , W)$. 
 
\begin{corollary} 
\label{filtrationclasses} 
Suppose we have a filtration patching datum $(\rho , W)$.  
Then again supposing $X$ is the homotopy  pushout of  
$U\leftarrow B^{\ast}\rightarrow B$ we get a map $X\rightarrow BGL(F)^+$ and classes  
in $H^{\ast}(X,k)$ for any class in $H^{\ast}(BGL(F),k)$. In particular for $\sigma : F\rightarrow \comx$  
we get regulator  
classes denoted $\widehat{c}^{\rm def}_p(\rho , W)\in H^{2p-1}(X,\comx / \Z )$.  
\end{corollary} 
 

\subsection{Comparison with the classes defined by patched connections}
 
We would like to compare these with the classes defined by the patched connections. 
As described at the beginning of this section, consider compact subsets $U_0\subset U$ and
$B_0\subset B$, retracts of the bigger subsets, such that $U_0$ is the complement of
an open tubular neighborhood of $D$ and $B_0$ is a smaller closed tubular neighborhood.
Consider $B_0^{\ast} \cong S\times [0,1]$, the closure of $X-U_0-B_0$. Thus, 
$X$ is obtained by glueing together $U_0$ and $B_0$ with the cylinder $B_0^{\ast}$.
In this way $X$  can be seen as a homotopy pushout. 

Recall that $S$ is an $S^1$ bundle over $D$.  
 
Suppose $F=\comx $ and we have a deformation patching datum $(\eta _U,\eta _B, \eta _{B^{\ast}})$.
Suppose also that the representations $\eta _U,\eta _B$ (resp. $\eta _{B^{\ast}})$) 
go into a finite rank group  
$GL_r(F) = GL_r(\comx )$ (resp. $GL_r(F[t]) = GL_r(\comx [t])$).  Then we can 
define a patched up connection as follows. Let $\eta _U$ and $\eta _B$ correspond to flat connections on  
$U_0$ and $B_0$. The deformation $\eta _{B^{\ast}}$ into 
$GL_r(\comx [t])$ can be evaluated at $t\in [0,1]$, via the evaluation map $e_t:\comx [t]\rightarrow \comx$.
This gives a family of representations in $GL_r(\comx )$, which may also be viewed as a 
family of flat connections on $S$ parametrized by $t\in [0,1]$. Taking the connection form
to be zero in the $dt$ direction gives a connection
over the 
cylinder $B^{\ast}_0=S\times [0,1]$, a connection which on the endpoints glues together with
the given flat bundles on $U_0$ or $B_0$. Putting them together we get a connection $\nabla ^{\rm def}$ on $X$.  
It has the property that in $U_0$ and $B_0$ it is flat, whereas in $B^{\ast}_0=S\times [0,1]$ it is 
flat along each $t$-level set $S \times \{ t \}$.  
 
Now note that the curvature form  $\Omega = \Omega _{\nabla ^{\rm def}}$ restricts to zero on the $t$-level sets. 
Since these have codimension $1$, it follows that in any local coordinates of the form $(t,x_i)$ 
where $x_i$ are local coordinates on $S$, all terms in $\Omega$ have 
a $dt$, that is there are no terms of the form $dx_i\wedge dx_j$. It follows that $\Omega \wedge \Omega = 0$. 
In particular the Chern forms of $\Omega$ vanish except maybe for the first one.  
 
The differential character $\widehat{c}_p(\nabla ^{\rm def})$  
associated to the connection $\nabla ^{\rm def}$ therefore projects to zero in  
the closed $2p$-forms whenever $p > 1$, so it defines a class in $H^{2p-1}(X, \comx / \Z )$ for any $p>1$. 
 
The two things we need to know are resumed in the following lemmas.  
 
\begin{lemma} 
\label{kclassdefclass}
The class $\widehat{c}_p(\nabla ^{\rm def})$ obtained by using the deformed connection on the cylinder,  
is equal to the regulator class 
$\widehat{c}_p^{\rm def}(\eta _U, \eta _B, \eta _{B^{\ast}})$
defined using the deformation theorem in $K$-theory.  
\end{lemma} 
\begin{proof}
It suffices to prove this for the universal case where 
$$
U =  BGL(F) \cup ^{BGL(F[t] \times \{ 0\} } \left( BGL(F[t])\times [0,1) \right) ,
$$
$$
B =  BGL(F) \cup ^{BGL(F[t] \times \{ 1 \} } \left( BGL(F[t] )\times (0,1] \right) ,
$$
$$
B^{\ast} = BGL(F[t] )\times (0,1)
$$
and $X = BGL(F)_{\rm def}$ is the homotopy pushout.
In this case we know that the spaces $U$, $B$ and $B^{\ast}$ have the same homology, 
so the connecting map in Mayer-Vietoris is trivial.
Thus we have an exact sequence
$$
0\rightarrow H^{\ast}(X, \comx / \Z ) \rightarrow H^{\ast}(U, \comx / \Z ) \oplus H^{\ast}(B, \comx / \Z ) \rightarrow 
H^{\ast}(B^{\ast}, \comx / \Z ) \rightarrow O.
$$
Now, we the class defined by pullback under the map using the deformation triple in $K$-theory, restricts on $U$ and $B$
to the standard class. The same is true for the class defined by the previous construction. Thus, they are equal
in $H^{\ast}(X, \comx / \Z )$.
\end{proof}
 
\begin{lemma} 
\label{defclasspatch}
Starting with $(\rho , W)$, do the procedure of \S \ref{deformation-filtration} to get $\eta _U$,
$\eta _B$ and $\eta _{B^{\ast}}$. The class $\widehat{c}_p(\nabla ^{\rm def})$ 
defined using this deformation triple, is equal to the class defined in \S \ref{patchsmoothdiv} using a 
patched connection $\nabla ^{\#}$ for $(\rho , W)$.
\end{lemma} 
\begin{proof}
The two classes come from connections $\nabla ^{\rm def}$ and
$\nabla ^{\#}$ respectively. Both of these connections
are compatible with the pre-patching collection associated to
$(\rho , W)$ in \S \ref{patchsmoothdiv}.  By 
Lemma \ref{nilinv} the classes are the same.
\end{proof}

With these two lemmas we get that for $p> 1$ the patched connection class is the same as the class defined 
using $K$-theory as above.  
 
\begin{corollary}
\label{patchingsame}
Suppose we are given a filtration triple. Then the regulator classes defined on the one hand
using the map $X\rightarrow BGL(F) _{\rm def}^+$ obtained by using the associated deformation triple; and
on the other hand using the patching construction of Corollary \ref{patchedclass}, coincide.
\end{corollary}
\begin{proof}
We have
$$
\widehat{c}_p^{\rm def}(\eta _U, \eta _B, \eta _{B^{\ast}}) = 
\widehat{c}_p(\nabla ^{\rm def}) = \widehat{c}_p(\nabla ^{\#}) = \widehat{c}_p(\rho /X).
$$
The first equality is by
Lemma \ref{kclassdefclass}, 
the second equality by Lemma \ref{defclasspatch}, and the third is the definition of
$\widehat{c}_p(\rho /X)$. 
\end{proof}

\begin{remark}
\label{unclear-def}
The above argument is another place where it becomes unclear how to generalize our procedure
to the case of a normal crossings divisor. Near the codimension $k$ pieces of the stratification 
of $D$, the ``collar'' looks like $S\times [0,1]^k$. However, if we envision a $k$-variable
deformation, then the argument saying that the higher Chern forms vanish, no longer works.
This then is another reason why we are restricting to the case of a smooth divisor in the present paper.
\end{remark}


\section{Hermitian $K$-theory and variations of Hodge structure}\label{hermitian K-theory}


In order to prove that the volume invariants vanish in the case of variations of Hodge structure,
Reznikov used a direct calculation of the space of invariant polynomials on a group of Hodge type.
In order to apply this idea to the extended regulators, we use a variant of the previous 
deformational construction for hermitian $K$-theory.


\subsection{Hermitian $K$-theory}


Start by recalling some of the basics of hermitian $K$-theory. See for example \cite{Karoubi}.

We work with commutative rings $A$ with involution $a\mapsto \overline{a}$ preserving the product. 
The basic example is $\comx$ with
the complex conjugation involution. 
Given an $A$-module $V$ we denote by $\overline{V}$ the
same set provided with the conjugate $A$-module structure.
On the other hand, denote by $V^{\ast}$ the usual dual module of a projective $A$-module. 
Note that we have a natural isomorphism 
$$
(\overline{V})^{\ast} \cong \overline{(V^{\ast})}
$$
and these will be indiscriminately noted $\overline{V}^{\ast}$. Either one may be viewed as the module of
antilinear homomorphisms $\lambda : V\rightarrow A$, that is such that $\lambda (av)= \overline{a}\lambda (v)$.

An {\em hermitian pairing} is a morphism 
$$
h : V\rightarrow \overline{V}^{\ast}
$$
which may be interpreted as a form 
$$
u,v\mapsto \langle u,v\rangle _h =h(u)(\overline{v}) \in A
$$
satisfying the properties
$$
\langle a u , v\rangle _h = a \langle u , v\rangle _h, \;\;\; \langle u , av\rangle _h = \overline{a} \langle u , v\rangle _h .
$$
Fix $\epsilon = \pm 1$. 
An {\em $\epsilon$-hermitian module} over $A$ is a pair $(V,h)$ consisting of a projective $A$-module $V$ provided 
with an hermitian pairing $h$ such that 
$$
\langle v, u\rangle _h = \epsilon \overline{ \langle u , v\rangle _h}.
$$

If there exists $i\in A$ with $i ^2= 1$ and $\overline{i}= -i$ and $h$ is an $\epsilon$-hermitian pairing then
$ih$ is a $-\epsilon$-hermitian pairing. So, in this case the distinction between the two values of $\epsilon$ disappears.
This happens for the rings we consider.  

If $V$ is any $A$-module then the {\em hyperbolic $\epsilon$-hermitian $A$-module} is defined by 
$H^{\epsilon}(V) = V\oplus \overline{V}^{\ast}$ with $h$ interchanging the factors with a sign determined by $\epsilon$. 
Put $H^{\epsilon}_{n,n}:= H_{\epsilon}(A^n)$. 

Let $O(V,h)$ be the group of automorphisms of the $\epsilon$-hermitian $A$-module 
$(V,h)$. Let $O^{\epsilon}_{n,n}(A) := O(H^{\epsilon}_{n,n})$ and let $O^{\epsilon}_{\infty , \infty}(A)$ 
be the direct limit of these groups for the natural inclusion maps as $n\rightarrow \infty$. It has a perfect commutator subgroup just as 
is the case for $GL_{\infty}(A)$, so we can make the {\em plus construction} 
$$
BO^{\epsilon}_{\infty , \infty}(A) ^+ .
$$
Karoubi defines the Quillen-Milnor $L$-groups by 
$$
L^{\epsilon}_n(A) := \pi _n BO^{\epsilon}_{\infty , \infty}(A) ^+.
$$
Recall that 
$$
K_n(A):= \pi _nBGL_{\infty}(A) ^+.
$$

The hyperbolic construction gives a map
$$
H : GL_{\infty}(A) \rightarrow O^{\epsilon}_{\infty , \infty}(A), \;\; \mbox{hence}\;\;
H^+ : BGL_{\infty}(A)^+ \rightarrow BO^{\epsilon}_{\infty , \infty}(A) ^+,
$$
and on the other hand forgetting the hermitian form gives a map
$$
F:  O^{\epsilon}_{\infty , \infty}(A) \rightarrow GL_{\infty}(A), \;\; \mbox{hence}\;\;
F^+ : BO^{\epsilon}_{\infty , \infty}(A) ^+ \rightarrow BGL_{\infty}(A)^+ .
$$
These give maps between the $K$-groups and the $L$-groups: 
$$
H: K_n(A)\rightarrow L^{\epsilon}_n(A),
$$
$$
F: L^{\epsilon}_n(A)\rightarrow K_n(A).
$$

Karoubi considers the cokernel 
$$
W^{\epsilon}_n(A):= {\rm coker}\left( H: K_n(A)\rightarrow L^{\epsilon}_n(A) \right) ,
$$
and on \cite[page 392, Corollaire 5.8]{Karoubi} he defines $\overline{W}^{\epsilon}_n(A)$ by inverting the prime $2$. 
The polynomial ring $A[x]$ has an involution extending that of $A$, defined by $\overline{x}= x$. 
One  of his main results is the following:

\begin{theorem}[{Karoubi \cite{Karoubi}, Corollaire 5.11}]
\label{karoubi}
The inclusion $A\rightarrow A[x]$ induces isomorphisms $\overline{W}^{\epsilon}_n(A) \cong \overline{W}^{\epsilon}_n(A[x])$. 
\end{theorem}

The following corollary was undoubtedly considered obvious in \cite{Karoubi} but needs to be stated.

\begin{corollary}
\label{Linvariance}
Letting $\Z ':= \Z [\frac{1}{2}]$ the inclusion $A\rightarrow A[x]$ induces isomorphisms on $L$-theory 
$$
L^{\epsilon}_n(A)\otimes \Z ' \stackrel{\cong}{\rightarrow}  L^{\epsilon}_n(A[x]) \otimes \Z ' .
$$
\end{corollary}
\begin{proof}
Evaluation at $0$ gives a splitting $A\rightarrow A[x]\stackrel{e_0}{\rightarrow} A$,
compatible with the hermitian structure. It follows that the morphism 
$$
L^{\epsilon}_n(A) \rightarrow  L^{\epsilon}_n(A[x]) 
$$
is  a split inclusion. Now we have a diagram with horizontal right exact sequences 
$$
\begin{array}{cccccc}
K_n(A)\otimes \Z ' & \rightarrow & L^{\epsilon}_n(A)\otimes \Z ' & \rightarrow &  \overline{W}^{\epsilon}_n(A) & \rightarrow 0 \\
\downarrow & &  \downarrow & & \downarrow & \\
K_n(A[x])\otimes \Z ' & \rightarrow & L^{\epsilon}_n(A[x])\otimes \Z ' & \rightarrow &  \overline{W}^{\epsilon}_n(A[x]) & \rightarrow 0 
\end{array}
$$
where the left vertical arrow is an isomorphism by the fundamental homotopy invariance theorem in $K$-theory,
the middle arrow is a split inclusion, and the right vertical arrow is an isomorphism by Theorem \ref{karoubi}. 
It follows that the middle vertical arrow is surjective, so it is an isomorphism.
\end{proof}

\begin{corollary}
\label{hermitianhomotopy}
For any ring with involution $A$, the map 
$$
BO^{\epsilon}_{\infty , \infty}(A) ^+ \rightarrow BO^{\epsilon}_{\infty , \infty}(A[x]) ^+
$$
induces a homotopy equivalence after localizing away from the prime $2$ (in particular, for rational homotopy theory). The same is true of the evaluation maps
$$
e_{0}, e_{1} : BO^{\epsilon}_{\infty , \infty}(A[x]) ^+ \rightarrow BO^{\epsilon}_{\infty , \infty}(A) ^+
$$
\end{corollary}

Let $BO^{\epsilon}_{\infty , \infty}(A) ^+_{\rm def}$ denote the homotopy pushout of the evaluation maps $e_0, e_1$
appearing in the previous corollary. Then also the map 
$$
BO^{\epsilon}_{\infty , \infty}(A) ^+ \rightarrow BO^{\epsilon}_{\infty , \infty}(A)_{\rm def} ^+
$$
is an equivalence after localizing away from $2$ and in particular in rational homotopy theory. 

Note that for each evaluation map $e_i$, $i=0,1$ there is a commutative diagram
$$
\begin{array}{ccc}
BO^{\epsilon}_{\infty , \infty}(A[x] ) ^+ & \rightarrow & BGL_{\infty}(A[x])^+ \\
\downarrow & & \downarrow \\
BO^{\epsilon}_{\infty , \infty}(A ) ^+ & \rightarrow & BGL_{\infty}(A)^+
\end{array}
$$
where the vertical maps are the evaluation maps. This gives a commutative diagram of homotopy pushout squares
which we don't write down, in which the pushout map is 
$$
BO^{\epsilon}_{\infty , \infty}(A)_{\rm def} ^+  \rightarrow  BGL_{\infty}(A)_{\rm def}^+
$$
which is compatible with the rest. 
 

\subsection{Hermitian deformation patching}


We now apply this to the case $A=\comx$ with the involution being complex conjugation. 
Since $i\in \comx$ by the above remark 
the choice of $\epsilon$ doesn't matter and we now take $\epsilon = 1$ and drop it from notation. 

The group $O_{n,n}(\comx )$ is
more commonly known as $U(n,n)$, the unitary group of the hermitian form of signature $n,n$. 
This is because the natural hermitian form on hyperbolic
space $H(\comx ^n)$ has signature $(n,n)$. Thus
$$
O_{\infty , \infty }(\comx ) = \lim _{\rightarrow} U(n,n).
$$
Note also that for any $p,q$ we have $U(p,q)\subset U(n,n)$ for $n\geq {\rm max}(p,q)$ so we can also write
$$
O_{\infty , \infty }(\comx ) = \lim _{\rightarrow} U(p,q).
$$
So, if $(X,x)$ is a path-connected pointed space with a representation $\rho : \pi _1(X,x) \rightarrow U(p,q)$ for some $p,q$,
then we obtain a map 
$$
X\rightarrow BO_{\infty , \infty}(\comx ) \rightarrow BO_{\infty , \infty}(\comx ) ^+.
$$

The patching construction as previously done applies in this case too. 

For a given $p,q$ let $V$ be the $\comx$-vector space with hermitian form $h$ of signature $p,q$. 
Let $V[t]:= V\otimes _{\comx}\comx [t]$ be its extension of scalars to $\comx [t]$. 
Let $O_{p,q}(\comx [t])$ denote the group of hermitian automorphisms of $V[t]$.
For $p=q=n$ this coincides with the previous notation $O_{n,n}(\comx [t])$ and 
For any $n\geq {\rm max}(p,q)$ we have an inclusion $O_{p,q}(\comx [t])\subset O_{n,n}(\comx [t])$
obtained by direct sum with a form of signature $n-p,n-q$. 

Suppose we have 
a diagram of pointed path-connected spaces
$$ 
(U,p)\stackrel{b}{\leftarrow} (B^{\ast}, p) \stackrel{c}{\rightarrow} (B, p) 
$$ 
together with representations  
$$ 
\eta _U : \pi _1(U, p) \rightarrow U(p,q), 
$$ 
$$ 
\eta _B : \pi _1(B, p) \rightarrow U(p,q), 
$$ 
and 
$$ 
\eta _{B^{\ast}} : \pi _1(B^{\ast},p) \rightarrow O_{p,q}(\comx [t]) 
$$ 
such that  
$$ 
e_1\circ \eta _{B^{\ast}} = b^{\ast}(\eta _U) \;\;\mbox{and} \;\;  
e_0\circ \eta _{B^{\ast}} = c^{\ast}(\eta _B) . 
$$ 
In other words, we again have representations on $U$ and $B$, plus a deformation between their 
restrictions to $B^{\ast}$.  
We call this an {\em hermitian deformation triple}. 

As before, we suppose given the subsets $U_0$, $B_0$ and $B^{\ast}_0 \cong S\times [0,1]$, and
$X:= U\cup ^{S\times [0,1]}B$ is the homotopy pushout, so we obtain a map 
$$
X\rightarrow BO_{\infty , \infty}(\comx )_{\rm def} ^+.
$$

\begin{lemma}\label{classifyingmap}
Composing the above representations with the inclusions $U(p,q)\subset GL(p+q,\comx )$ or
$O_{p,q}(\comx [t]) \subset GL(p+q, \comx [t])$ we obtain from our hermitian deformation triple 
a usual deformation patching datum in the previous sense. This in turn gives a map 
$$
X\rightarrow BGL(\comx )_{\rm def}^+ ,
$$
which is homotopy equivalent to the composition of
$$
X\rightarrow BO_{\infty , \infty}(\comx )_{\rm def} ^+ \stackrel{F}{\rightarrow}
BGL(\comx )_{\rm def}^+ .
$$
\end{lemma}
\begin{proof}
This comes from the compatibility of the homotopy pushout squares used to define
$BO_{\infty , \infty}(\comx )_{\rm def} ^+ $ and 
$BGL(\comx )_{\rm def}^+$. 
\end{proof}

Now the key part of the present argument comes from Reznikov's fundamental observation about the cohomology degrees
of generators of the cohomology theories on both sides.
Recall that the {\em Borel volume regulators} are classes 
$$
r^{\rm Bor}_p \in H^{2p-1}(BGL(\comx )^+, \R ).
$$
These correspond to the imaginary parts of the $\comx / \Z$ regulators we are studying. 

We can repeat all the constructions in \S \ref{kclassdef-sec} and in the present \S \ref{hermitian K-theory} for the special linear group $SL(\comx)$.
As in \cite[p.377, \S 2.7]{Re2}, we will eventually reduce to the case when we look
at $SL_r(\comx)$-valued representations. Thus Lemma \ref{classifyingmap} will give us maps
$$
X\rightarrow BSL(\comx )_{\rm def}^+ ,
$$
homotopy equivalent to the composition of
$$
X\rightarrow BSO_{\infty , \infty}(\comx )_{\rm def} ^+ \stackrel{F}{\rightarrow}
BSL(\comx )_{\rm def}^+ .
$$

\begin{lemma}
For any $p>1$ the pullback of $r^{\rm Bor}_p$ via the map 
$BSO_{\infty , \infty}(\comx )^+ \rightarrow
BSL(\comx )^+$ is zero. 
\end{lemma}
\begin{proof}
It suffices to show this for any finite stage $SU(p,q) \rightarrow SL_r( \comx )$. 
Then, Reznikov's argument, basically by observing that there are no $S(U(p)\times U(q))$-invariant
polynomials on $SU(p,q)$, gives the statement \cite{Re2}. 
\end{proof}

Since $BSL(\comx )^+\rightarrow BSL(\comx )_{\rm def}^+$ induces an isomorphism on rational homology,
the volume invariant extends to an invariant denoted also $r^{\rm Bor}_p$ on $BSL(\comx )_{\rm def}^+$.

\begin{corollary}\label{zerohomology}
For any $p>1$ the pullback of $r^{\rm Bor}_p$ via the map 
$BSO_{\infty , \infty}(\comx )_{\rm def}^+ \stackrel{F}{\rightarrow}
BSL(\comx )_{\rm def}^+$ is zero. 
\end{corollary}
\begin{proof}
The map is the same as in the previous lemma, on rational cohomology. 
\end{proof}

\begin{corollary}\label{volumehermtrip}
Given an hermitian deformation triple, the associated volume invariant $Vol^{\rm def}_{2p-1}(\eta _U, \eta _B, \eta _{B^{\ast}})$ 
is zero for any $p>1$. 
\end{corollary}
\begin{proof}
Recall that $Vol_{2p-1}(\eta _U, \eta _B, \eta _{B^{\ast}})$  is,
by definition, the pullback of $r^{\rm Bor}_p$ via the map $X\rightarrow BGL(\comx )^+_{\rm def}$ obtained
by deformation patching. This map is shown
in Lemma \ref{classifyingmap} to factor through $BO_{\infty , \infty}(\comx )_{\rm def}^+$. 
Then apply Corollary \ref{zerohomology}, using the reduction to $SL$ and $SO_{\infty , \infty}$ mentioned above,
from \cite[p.377, \S 2.7]{Re2}. 
\end{proof}

\subsection{An hermitian deformation triple associated to a VHS}

Consider a representation $\rho$ underlying 
a complex variation of Hodge structure, with unipotent monodromy along an irreducible smooth divisor $D$. 
In this case there is a VMHS $(V,W,F, \tilde{F},\langle \cdot , \cdot \rangle )$ on the divisor component.  
We don't need to know about the Hodge filtrations $F$ and $\tilde{F}$. 
The basic information we need to know about the weight filtration and the hermitian form 
is what is given by the $1$-variable nilpotent and SL2--orbit theorems (see \cite{Sch}). 
Look at
the data $(V,N, \langle \cdot , \cdot \rangle )$ where $V$ is the vector space, $N = \log \rho (\gamma )$ is the 
logarithm of the
monodromy around the divisor $D$, and $\langle \cdot , \cdot \rangle $ is the
flat indefinite hermitian form preserved by $\rho$.
We normalize to suppose that $\langle \cdot , \cdot \rangle $ is hermitian symmetric, rather than hermitian
antisymmetric; by multiplying by $i = \sqrt{-1}$ we can always assume this.

The one-variable nilpotent and SL$2$--orbit theorems imply 
that this triple $(V,N, \langle \cdot , \cdot \rangle )$ is a direct sum of
standard objects. The standard objects are
symmetric powers of the standard $2$--dimensional case where $V$ has basis $e_1,
e_2$, with $Ne_1 = e_2$ and $N e_2 = 0$;
and with $\langle e_i,e_i \rangle  = 0$ but $\langle  e_1,e_2 \rangle  = 1$.   

For the standard object of rank $2$, the monodromy
weight filtration has graded quotients
$Gr_1$ corresponding to $e_1$ and $Gr_{-1}$ corresponding to $e_2$, and the form 
$$
(u,v)\mapsto \langle u, Nv \rangle 
$$ 
is positive definite on $Gr_1$.
The zeroth symmetric power is just the case $N=0$. The $k$-th symmetric power of the
standard object has
basis vectors $e_0,...,e_k$ with $Ne_i = e_{i+1}$ and $\langle e_i,e_j \rangle  = 0$ unless $i+j=k$
in which case it is $1$.
In this case the monodromy weight filtration
puts $e_i$ in degree $k - 2i$,
going from $e_0$ in degree $k$ to $e_k$ in degree $-k$.

So in general our $V$ will be a direct sum of these kinds of things, and the
full monodromy representation of the
neighborhood of $D$ will preserve the monodromy weight filtration. 
Each of the standard objects comes with a splitting of the monodromy weight filtration,
so taking the direct sum of these splittings 
allows us to choose an isomorphism $V \cong Gr^W(V)$ or equivalently an expression
$$
V =\bigoplus V_k
$$
with $V_k$ corresponding
to the $Gr^W_k(V)$. Then $N: V_k \rightarrow V_{k-2}$, and this polarizes the hermitian form
induced by $\langle \cdot , \cdot \rangle$ on $Gr^W_k(V)$ as in \cite{Sch}.
The splittings of the standard objects relate the form $\langle \cdot , \cdot \rangle $  on the original 
vector space with
the induced form on the associated-graded pairing $Gr^W_k$ with $Gr^W_{-k}$, so the same 
is true of our splitting of $V$: the form $\langle \cdot , \cdot \rangle $ on $V$
is the same as the induced form on $Gr^W(V)$, and in terms of the decomposition of $V$ it pairs $V_k$ with
$V_{-k}$.
The full monodromy representation is upper triangular for this block
decomposition.

\begin{proposition}\label{hermitiantriple}
With notations as above,
given a VHS on $U$ with unipotent monodromy around $D$, we can construct an hermitian deformation triple on $X$.
\end{proposition}
\begin{proof}
The representation $\eta _U$ is given by $\rho$, and $\eta _B$ is given by the associated-graded $Gr^W(\rho )$
transported to a representation on $V$ by the splitting. Note that $\eta _B$ still takes values in the 
unitary group $U(p,q)$ of $\langle \cdot , \cdot \rangle$ on $V$. Indeed, $\eta _B$ is a direct sum of 
representations which preserve the form on $Gr^W_k$ obtained by the polarization using $N$. On anything of
the form $Gr^W_k\oplus Gr^W_{-k}$ the form $\langle \cdot , \cdot \rangle$ is of hyperbolic type using
the polarization forms on the two pieces, so $\eta _B$ preserves the form $\langle \cdot , \cdot \rangle$
on each piece $Gr^W_k\oplus Gr^W_{-k}$.

Define the deformation
$\eta_{B^{\ast}}$ as follows (this is basically the same as in Proposition \ref{trivext} and  \S \ref{deformation-filtration} 
above):
for $t\in  \R$, let $T_t$ be the automorphism of $V$ which acts by multiplication
by $t^k$ on $V_k$.
Then conjugation with $T_t$ gives an action on $GL(V)$ which multiplies the
block diagonal pieces by $1$ and the
strictly upper triangular pieces by some positive powers of $t$. Thus, on an
upper triangular monodromy representation
$\rho |_{\pi _1(B^{\ast})}$
it extends to the case $t=0$ giving a family of representations $\eta _{B^{\ast}}:=Ad(T_t)(\rho|_{\pi _1(B^{\ast})})$
defined even for $t=0$
and at $t=0$ the image is the associated-graded representation $\eta _B$. This gives the required
deformation. Notice that $T_t$ preserves the form $\langle \cdot , \cdot \rangle $, because if
$v_i \in V_i$ and $v_j \in V_j$
then $\langle  v_i , v_j  \rangle  = 0$ unless $i+j=0$, and if $i+j=0$ then 
$$
\langle T_tv_i, T_tv_j \rangle  = \langle t^i v_i , t^{-i}v_j  \rangle  = \langle  v_i , v_j  \rangle .
$$
Thus the conjugated representations $Ad(T_t)(\rho |_{\pi _1(B^{\ast})})$ are all in the unitary
group of $(V,\langle \cdot  , \cdot \rangle )$ (even for $t=0$ by continuity). Hence $\eta _{B^{\ast}}$ takes values in 
$O_{p,q}(\comx [t])$. This completes the construction of the hermitian deformation triple. 
\end{proof}

In this situation $X$ is identified with the homotopy pushout $U_0\cup ^{S\times [0,1]}B_0$, and we will 
obtain a map 
$$
f_\rho:X\rightarrow BO_{\infty , \infty}(\comx )_{\rm def} ^+.
$$

\begin{corollary}
\label{volumeVHSvanishes}
The pullback of the volume regulator $r^{\rm Bor}_p$ by the map $F\circ f_\rho$
is the same as $Vol_{2p-1}(\rho /X)$ and it vanishes. 
\end{corollary}
\begin{proof}
Consider the hermitian deformation triple obtained in Proposition \ref{hermitiantriple}. This gives the map $f_{\rho}$.
The associated deformation patching datum is the same as the one used to define $\widehat{c}_p^{\rm def}(\rho ,W)$,
because the construction in the proof of Proposition \ref{hermitiantriple} complexifies to the same one as in
\S \ref{deformation-filtration}. Therefore
$(F\circ f_{\rho})^{\ast}(r^{\rm Bor}_p) = Vol _{2p-1}(\rho /X)$. 
As in Corollary \ref{volumehermtrip}, 
the proof now follows from Lemma \ref{classifyingmap} and Corollary \ref{zerohomology}.
\end{proof}

\section{The generalization of Reznikov's theorem}
\label{genrezn}

We can now give the generalization of Reznikov's theorem for canonical extensions in the
case of a smooth divisor. 

\begin{theorem}
\label{extendedReznikov}
Suppose $X$ is a smooth quasiprojective variety, with $D\subset X$ an irreducible
closed smooth divisor. Suppose $\rho : \pi _1(X-D)\rightarrow GL_r(\comx )$
is a representation such that $\rho (\gamma )$ is unipotent for $\gamma $ the loop going around 
$D$. Then the extended regulator 
$$
\widehat{c}_p(\rho /X ) \in H^{2p-1}(X,\comx / \Z )
$$
defined using the patched connection in Corollary \ref{patchedclass}, is torsion.
\end{theorem}
\begin{proof}
By the rigidity result \ref{rigidity}, the regulator doesn't change if we deform $\rho$. 
Thus, we may assume that $\rho$ is defined over an algebraic number field $F$.
The regulator is a pullback of a class via the map
$$
\xi _{\rho}: X\rightarrow BGL(F)^+_{\rm def}.
$$
Suppose $\sigma : F\hookrightarrow \comx$ is any embedding. Composing, we get a map
$$
X\rightarrow BGL(\comx )^+_{\rm def}.
$$
The pullback of the volume regulator by this map, is a class $Vol_{2p-1}(\rho ^{\sigma})\in H^{2p-1}(X,\R )$.
This class is independent of deformations of $\rho ^{\sigma}$ within representations which are unipotent along $D$,
by Theorem \ref{rigidity}. As in \cite[p.377, \S 2.7]{Re2}, it suffices to consider the case when the 
representation takes values in $SL_r(\comx)$. Indeed, one can take a $r$-fold covering $(Y,D_Y)\rar (X,D)$ 
such that the pullback of the canonical extension is the canonical extension of a tensor product of a flat 
line bundle and a unimodular flat bundle. The additivity of volume regulators and injectivity of the 
cohomology $H^j(X,\R)\rar H^j(Y,\R)$ and
the reasoning in {\em loc.\ cit} says that it suffices to prove the theorem for $SL_r$-valued representations. 
Also note that the constructions in \S \ref{kclassdef-sec} and \S \ref{hermitian K-theory} 
hold verbatim if we look at the special linear subgroups.

Mochizuki proves in \cite{Mochizuki} that $\rho ^{\sigma}$ may be deformed to a complex
variation of Hodge structure. When $X$ is smooth and projective with a smooth divisor $D$,
this can also be obtained using Biquard's earlier version of the theory for this case \cite{Bi}. 
This deformation preserves the condition of unipotence at infinity, since it
preserves the trivial parabolic structure of the Higgs bundle, and the Higgs field is multiplied by 
$t\rightarrow 0$ so if the eigenvalues are zero to begin with, then they are zero in the deformation.
On the other hand, by Corollary \ref{volumeVHSvanishes}, the extended
volume regulator vanishes for a complex variation of Hodge structure. Thus, $Vol_{2p-1}(\rho ^{\sigma})= 0$.
We now apply Reznikov's argument: by Borel's theorem, the classes $\sigma ^{\ast}(Vol _{2p-1})$ generate
the real cohomology ring of  $BSL(F)^+$ or equivalently $BSL(F)^+_{\rm def}$. The fact that their pullbacks by
$\xi _{\rho}$ vanish, implies that $\xi _{\rho}$ induces the zero map on rational homology. This in turn
implies that the pullback by $\xi _{\rho}$ of the universal class in $H^{2p-1}(BSL(F)^+_{\rm def},\comx / \Z )$, 
is torsion. 
By Corollary \ref{patchingsame}, the pullback of this class is the same as the regulator we have defined using 
the patched connection. 
\end{proof}

\begin{corollary}
Suppose $X$ is a smooth projective variety over $\comx$, 
with $D\subset X$ an irreducible
closed smooth divisor. 
Suppose $\rho : \pi _1(X-D)\rightarrow GL_r(\comx )$
is a representation such that $\rho (\gamma )$ is unipotent for $\gamma $ the loop going around
$D$. Let $(E,\nabla )$ be the holomorphic bundle with flat connection on $X-D$ associated to $\rho$,
and let $E_X$ be the Deligne canonical extension to a holomorphic bundle on $X$ with logarithmic connection 
having nilpotent residue along $D$. Then the Deligne Chern class 
$$
c^{\cD}_p(E_X) \in H^{2p}_{\cD}(X, \Z (p))
$$
is torsion.
\end{corollary}
\begin{proof}
We have shown in Proposition \ref{gluedsecondary} that the regulator class $\widehat{c}_p(\rho )$ lifts the Deligne Chern class of the canonical extension.
Thus, Theorem \ref{extendedReznikov} implies that the Deligne Chern class of the canonical extension is torsion.
\end{proof}

Aside from the obvious problem of generalizing these results to the case of a normal-crossings divisor,
another interesting question is how to generalize Reznikov's other proof of his theorem 
\cite{Re2}. This passed through a direct calculation of Borel's volume invariants using the harmonic
map, instead of invoking deformation to a variation of Hodge structure. It would be interesting to see
how to do this calculation for the volume invariant over $X$,
using the harmonic map associated to $\rho$ on $X-D$. This might lead to a better way of treating the 
normal-crossings case. 

Another circle of questions clearly raised by Reznikov's result is to determine the
torsion pieces of these classes, for example is there some arithmetical construction of these?
Can one bound the torsion or construct coverings on which it vanishes?  



\begin{thebibliography}{AAAAA}




\bibitem[Bi]{Bi}
O. Biquard. {\em Fibr\'es de Higgs et connexions int\'egrables: le cas logarithmique (diviseur lisse).} 
Ann. Sci. Ecole Norm. Sup. (4)  30  (1997),  no. \textbf{1}, 41--96.

\bibitem [Bl]{Bl} S. Bloch, {\em Applications of the dilogarithm function in algebraic K-theory and algebraic geometry}, 
Int.Symp. on Alg.Geom., Kyoto, 1977, 103-114.

\bibitem[Bo]{Borel}
A. Borel. 
{\em Stable real cohomology of arithmetic groups}.
Ann. Sci. Ecole Norm. Sup. \textbf{7} (1974) (1975), 235--272. 

\bibitem[Bo2]{Borel2}
A. Borel.
{\em Stable real cohomology of arithmetic groups. II}. Manifolds and Lie groups (Notre Dame, Ind., 1980),
Progr. Math., \textbf{14}, Birkh\"{a}user, Boston, Mass., (1981), 21--55. 

\bibitem[Br]{Br}
J-L. Brylinski, {\em Comparison of the Beilinson-Chern classes with the Chern-Cheeger-Simons classes}.  
Advances in geometry,  95--105, Progr. Math., \textbf{172}, Birkh\"auser Boston, Boston, MA, 1999. 

\bibitem[Ch-Sm]{Ch-Si}J. Cheeger, J. Simons, {\em
Differential characters and geometric invariants}, Geometry and topology (College Park, Md., 1983/84), 
50--80, Lecture Notes in Math., \textbf{1167}, Springer, Berlin, 1985. 

\bibitem[Chn-Sm]{Chn-Si} S.S. Chern, J. Simons, {\em Characteristic forms and geometric invariants},  Ann. of Math. (2)  \textbf{99}  (1974), 48--69. 

\bibitem [De]{De} P. Deligne,  {\em Equations diff\'erentielles \`a points singuliers
reguliers}. Lect. Notes in Math. $\bf{163}$, 1970.

\bibitem[De2]{De2} P. Deligne, {\em La conjecture de Weil. I.} (French), Inst. Hautes \'Etudes Sci. Publ. Math. No. \textbf{43} (1974), 273--307.                       

\bibitem[De3]{De3} P. Deligne, {\em Letter to J.N. Iyer}, dated 16.11.2006.

\bibitem[De-Su]{De-Su} P. Deligne, D. Sullivan,  {\em Fibr\'es vectoriels complexes \'a groupe structural discret},  
C. R. Acad. Sci. Paris S\'er. A-B  281  (1975), no. \textbf{24}, Ai, A1081--A1083. 

\bibitem[DHZ]{DHZ} J. Dupont, R. Hain, S. Zucker, {\em Regulators and characteristic classes of flat bundles},  
The arithmetic and geometry of algebraic cycles (Banff, AB, 1998),  47--92, CRM Proc. Lecture Notes, \textbf{24}, Amer. Math. Soc., Providence, RI, 2000. 

\bibitem[Es]{Es} H. Esnault, {\em Characteristic classes of flat bundles}, Topology  27  (1988),  no. \textbf{3}, 323--352.


\bibitem[Es2]{Es2} H. Esnault, {\em Recent developments on characteristic classes of flat bundles on complex algebraic manifolds}, 
Jahresber. Deutsch. Math.-Verein.  98  (1996),  no. \textbf{4}, 182--191. 


\bibitem[Es-Co]{Es-Co} H. Esnault, K. Corlette, {\em Classes of Local  Systems of Hermitian vector spaces}, appendix to
 J. M. Bismut, {\em Eta invariants, differential characters and flat vector bundles}, Chinese Ann. Math. Ser. B  26  (2005),  no. \textbf{1}, 15--44.


\bibitem [Es-Vi]{Es-Vi} H. Esnault, E. Viehweg,  {\em Logarithmic De Rham complexes
and vanishing theorems}, Invent.Math., $\bf{86}$, 161-194, 1986.

\bibitem [Gi-So]{GilletSoule} H. Gillet, C. Soul\'e, {\em Arithmetic Chow groups and differential characters.}  
Algebraic $K$-theory: connections with geometry and topology (Lake Louise, AB, 1987),  29--68, 
NATO Adv. Sci. Inst. Ser. C Math. Phys. Sci., \textbf{279}, Kluwer Acad. Publ., Dordrecht, 1989. 

\bibitem[Gr-Gri] {Green-Griffiths} M. Green, P. Griffiths, {\em Hodge-theoretic invariants for algebraic cycles.}  
Int. Math. Res. Not.  2003,  no. \textbf{9}, 477--510.

\bibitem[Gri]{Griffiths} P. Griffiths, {\em Some results on algebraic cycles on algebraic manifolds.} 
 1969  Algebraic Geometry (Internat. Colloq., Tata Inst. Fund. Res., Bombay, 1968)  pp. 93--191 Oxford Univ. Press, London.

\bibitem[Gri-Ha]{GriffithsHarris}P. Griffiths, J. Harris, {\em Principles of algebraic geometry.} Reprint of the 1978 original. 
Wiley Classics Library. John Wiley and Sons, Inc., New York, 1994. xiv+813 pp.


\bibitem [Gk]{Grothendieck} A. Grothendieck, {\em Classes de Chern et repr\'esentations lin\'eaires des groupes discrets.} (French) 1968  
Dix Expos\'es sur la Cohomologie des Sch\'emas  p. 215--305 North-Holland, Amsterdam; Masson, Paris. 

\bibitem[Ka]{Karoubi}
M. Karoubi.
P\'eriodicit\'e de la $K$-th\'eorie hermitienne. {\em Algebraic $K$-theory, III: Hermitian $K$-theory and geometric applications 
(Proc. Conf., Battelle Memorial Inst., Seattle, Wash., 1972)},  {\sc Lecture Notes in Math.} {\bf 343} (1973), 301-411.

\bibitem[Ka2]{Ka2} M. Karoubi, {\em Classes caract\'eristiques de fibr\'es feuillet\'es, 
holomorphes ou alg\'ebriques.} 
Proceedings of Conference on Algebraic Geometry and Ring Theory in 
honor of Michael Artin, Part II (Antwerp, 1992).  $K$-Theory  8  (1994),  no. \textbf{2}, 153--211.

\bibitem[Ko]{Ko} S. Kobayashi, {\em Induced connections and imbedded Riemannian spaces},
Nagoya Math. J. \textbf{10} (1956), 15--25. 


\bibitem[Mo]{Mochizuki} T. Mochizuki,  {\em Kobayashi-Hitchin correspondence for 
tame harmonic bundles and an application}, 
Ast\'erisque No. \textbf{309} (2006). 


Preprint {\tt math.DG/0411300}.

\bibitem[Na-Ra]{Narasimhan} M. S. Narasimhan, S. Ramanan, {\em Existence of universal connections},  Amer. J. Math.  \textbf{83}  1961 563--572. 

\bibitem[Na-Ra2]{Narasimhan2} M. S. Narasimhan, S. Ramanan, {\em Existence of universal connections. II},  Amer. J. Math.  \textbf{85}  1963 223--231.

\bibitem[Qu]{Quillen} D. Quillen. {\em Higher algebraic $K$-theory I}, Algebraic $K$-Theory I (Battelle 1972), 
Lecture Notes in Math. \textbf{341}, Springer-Verlag (1973), 85--147. 

\bibitem [Re]{Re} A. Reznikov, {\em Rationality of secondary classes},  J. Differential Geom.  43  (1996),  no. \textbf{3}, 674--692. 


\bibitem [Re2]{Re2} A. Reznikov,  {\em All regulators of flat bundles are torsion}, Ann. of Math. (2) 141 (1995), no. \textbf{2}, 373--386.

\bibitem[Ro]{Rosenberg}
J. Rosenberg. {\em Algebraic $K$-Theory and its Applications}, Graduate Texts in Math. \textbf{147}, Springer-Verlag (1994). 

\bibitem[Sch]{Sch}
W. Schmid. 
{\em Variation of Hodge structure: the singularities of the period mapping.}
Invent. Math. \textbf{22} (1973), 211--319.


\bibitem[Sm-Su]{Simmons} J. Simons, D. Sullivan, {\em Axiomatic Characterization of Ordinary Differential Cohomology}, arXiv math.AT/0701077.

\bibitem[Si]{Si} C. Simpson, {\em Harmonic bundles on noncompact curves.}  J. Amer. Math. Soc.  3  (1990),  no. \textbf{3}, 713--770.

\bibitem[Si2]{Si2} C. Simpson, {\em Higgs bundles and local systems}, Inst. Hautes \'Etudes Sci. Publ. Math.  No. \textbf{75}  (1992), 5--95.

\bibitem[So]{Soule} C. Soul\'e, {\em Connexions et classes caract\'eristiques de Beilinson.} 
 Algebraic $K$-theory and algebraic number theory 
(Honolulu, HI, 1987),  349--376, Contemp. Math., \textbf{83}, Amer. Math. Soc., Providence, RI, 1989.

\bibitem[So2]{Soule2} C. Soul\'e, {\em Classes caract\'eristiques secondaires des fibr\'es plats.} 
S\'eminaire Bourbaki, Vol. 1995/96.  Ast\'erisque  No. 241  (1997), Exp. No. 819, \textbf{5}, 411--424. 

\bibitem[Sr]{Srinivas} V. Srinivas, {\em Algebraic $K$-theory.} Progress in Mathematics, \textbf{90}. Birkh\"auser Boston, Inc., Boston, MA, 1991. xvi+314 pp.


\bibitem[Su]{Sullivan} D. Sullivan, {\em Genetics of homotopy theory and the Adams conjecture.}  Ann. of Math. (2)  \textbf{100}  (1974), 1--79.
 
\bibitem[Ta]{Ta} S. Takizawa, 
{\em On the induced connexions},
Mem. Coll. Sci. Univ. Kyoto. Ser. A. Math. \textbf{30} 1957 105--118. 

\bibitem[Zu]{Zu} S. Zucker, {\em The Cheeger-Simons invariant as a Chern class.}  
Algebraic analysis, geometry, and number theory (Baltimore, MD, 1988),  397--417, Johns Hopkins Univ. Press, Baltimore, MD, 1989.

\end {thebibliography}
\end{document}